\documentclass[10pt,a4paper]{article}
\usepackage{latexsym}
\usepackage{amssymb}
\usepackage{amsmath}
\usepackage{amsthm}
\usepackage{thmtools}
\usepackage{cite}
\usepackage[all]{xy}
\usepackage{framed}
\usepackage{hyperref}
\usepackage{graphicx}
\usepackage{parskip}
\usepackage{hyperref}
\usepackage{enumitem}
\usepackage[a4paper]{geometry}

\begingroup
    \makeatletter
    \@for\theoremstyle:=definition,remark,plain\do{%
        \expandafter\g@addto@macro\csname th@\theoremstyle\endcsname{%
            \addtolength\thm@preskip\parskip
            }%
        }
\endgroup

\makeatletter
\xydef@\txt@ii#1{\vbox{\vspace*{-5pt}%
 \let\\=\cr
 \tabskip=\z@skip \halign{\relax\hfil\txtline@@{##}\hfil\cr\leavevmode#1\crcr}}}
\makeatother

\theoremstyle{definition}
\newtheorem{thm}{Theorem}[section]
\newtheorem{lem}[thm]{Lemma}
\newtheorem{cor}[thm]{Corollary}
\newtheorem{defn}[thm]{Definition}
\newtheorem{propn}[thm]{Proposition}
\newtheorem*{thm*}{Theorem}
\newtheorem{notn}[thm]{Notation}

\newtheorem{setup}[thm]{Setup}

\newtheorem*{nts}{Note to self}

\theoremstyle{remark}
\newtheorem{rk}[thm]{Remark}
\newtheorem{rks}[thm]{Remarks}

\newtheoremstyle{custthm}{\parskip}{}{\normalfont}{}{\bfseries}{.}{ }{\thmname{#1} \thmnote{#3}}
\theoremstyle{custthm}
\newtheorem*{letterthm}{Theorem}

\newcommand{\tensor}[1]{\underset{#1}{\otimes}}

\newcommand{\Aut}{\mathrm{Aut}}

\newcommand{\cpi}{\mathrm{cpi}^{\overline{k\Delta^+}}}
\newcommand{\cpisum}[2]{#1|^{#2}}
\newcommand{\cp}[2]{\underset{\langle #1,#2 \rangle}{*}}
\newcommand{\Xinn}[3]{\mathrm{Xinn}_{#1}({#2}; {#3})}

\newcommand\blfootnote[1]{%
  \begingroup
  \renewcommand\thefootnote{}\footnote{#1}%
  \addtocounter{footnote}{-1}%
  \endgroup
}

{%
\begin{framed}\begin{nts}%
}%
{%
\end{nts}\end{framed}%
}

\begin{document}

\numberwithin{equation}{section}
\binoppenalty=\maxdimen
\relpenalty=\maxdimen

\title{Maximal prime homomorphic images of mod-$p$ Iwasawa algebras}
\author{William Woods}
\date{\today}
\maketitle
\begin{abstract}
Let $k$ be a finite field of characteristic $p$, and $G$ a compact $p$-adic analytic group. Write $kG$ for the completed group ring of $G$ over $k$. In this paper, we describe the structure of the ring $kG/P$, where $P$ is a minimal prime ideal of $kG$. We give an explicit isomorphism between $kG/P$ and a matrix ring with coefficients in the ring $(k'G')_\alpha$, where $k'/k$ is a finite field extension, $G'$ is a large subquotient of $G$ with no finite normal subgroups, and $(-)_\alpha$ is a ``twisting" operation that preserves many desirable properties of the ring structure. We demonstrate the usefulness of this isomorphism by studying the correspondence induced between certain ideals of $kG$ and those of $(k'G')_\alpha$, and showing that this preserves many useful ``group-theoretic" properties of ideals, in particular almost-faithfulness and control by a closed normal subgroup.

\end{abstract}

\newpage
\tableofcontents

%
%
%
%
%
%

\newpage

\section*{Introduction}

Let $G$ be a compact $p$-adic analytic group \cite[Definition 8.14]{DDMS} and $R$ a commutative pseudocompact ring \cite[\S 1]{brumer}. We will write $RG$ or $R[[G]]$ for the completed group ring of $G$ over $R$,
$$RG := \underset{N}{\underleftarrow{\lim}}\; R[G/N],$$
where this inverse limit ranges over all open normal subgroups $N$ of $G$, and $R[G/N]$ denotes the ordinary group ring of the (finite) group $G/N$ over $R$. Particularly in the cases when $R = \mathbb{F}_p$ and $R = \mathbb{Z}_p$, the topological ring $RG$ is often known as the \emph{Iwasawa algebra} of $G$. We give an explicit universal property for $RG$ in Lemma \ref{defn: u.p. of group rings}; in particular, continuous $R$-valued representations of $G$ are characterised by modules over $RG$ (analogously to the defining property of $R[G]$ when $G$ is a finite group).

\textbf{Throughout this introduction, we will take $k$ to be a finite field of characteristic $p$.}

It is already known \cite[Theorem 9.2]{ardakovJAG} that $kG$ is a prime ring if and only if $G$ contains no non-trivial finite normal subgroups; however, in general, $kG$ will have nonzero minimal (two-sided) prime ideals $P$. Our main objects of study in this paper are the prime quotients $kG/P$.

Due to a result of Ardakov \cite{ardakovGMJ}, these minimal primes $P\lhd kG$ are intimately connected with the minimal primes $\mathfrak{p}\lhd k\Delta^+$, where
$$\Delta^+ = \{x\in G \,\big|\, o(x) < \infty, [G:\mathbf{C}_G(x)]<\infty\}$$
is the \emph{finite radical} of $G$, as defined in the text just before \cite[Lemma 4.1.6]{passmanASGR}. (See \cite[\S 1]{woods-struct-of-G} for the properties of $\Delta^+$ in the case when $G$ is profinite: in particular, it is the unique largest finite normal subgroup of $G$.) We review the details of this connection in Lemma \ref{lem: definition of e}.

Until now, little was known about the structure of the rings $kG/P$, and all previously known structural results were underpinned by the following special case. The following theorem assumes familiarity with the definition of a \emph{uniform} group, which may be found in \cite[Definition 4.1]{DDMS}.

\begin{thm*}{\cite[Proposition 10.1(a), Theorem 9.2]{ardakovJAG}}
Let $G$ be a compact $p$-adic analytic group. Suppose that $G = F\times U$, where $F$ is a finite group whose order is not divisible by $p$, and $U$ is a uniform group. Let $P$ be a minimal prime ideal of $kG$. Then $P = \mathfrak{p}kG$ for some prime $\mathfrak{p}$ of $k\Delta^+$; and there exist a finite field extension $k'/k$ and an integer $t\geq 1$ such that $kG/P \cong M_t(k'U)$.\qed
\end{thm*}

This theorem reduces the study of the prime spectrum of $kG$ to that of $M_t(k'U)$, and hence by Morita equivalence to that of $k'U$. This latter ring is in general far better understood, as uniform groups and their completed group rings have a particularly nice structure: see, for example, the survey paper \cite{ardakovbrown} or the text \cite{DDMS} for many results along these lines.

The results of this paper are a broad generalisation of this theorem. Given an arbitrary compact $p$-adic analytic group $G$ and a minimal prime $P$ of $kG$, our results explicitly relate the ideal lattice of $kG/P$ to that of a certain crossed product of $k'U$ with a finite group, where $U$ is a large pro-$p$ subquotient of $G$ (which may be taken to be e.g. uniform as required). We also show that this ideal correspondence preserves many desirable properties, such as primality. We will not rely on the special case above, although some elements of the proofs are similar.

\textbf{For the remainder of this introduction, we fix a compact $p$-adic analytic group $G$ with finite radical $\Delta^+$. As before, $k$ will always denote a finite field of characteristic $p$.}

First, we begin with a generalisation of the above theorem.

\begin{letterthm}[A]
Suppose that $G/\Delta^+$ is pro-$p$. Take a minimal prime $\mathfrak{M}$ of $kG$, and suppose that $\mathfrak{M}\cap k\Delta^+$ remains prime as an ideal of $k\Delta^+$. Then there exist a positive integer $t$, a finite field extension $k'/k$, and an isomorphism
\pushQED{\qed}
\begin{equation*}
\psi: kG/\mathfrak{M} \to M_t\Big( k'[[G/\Delta^+]]\Big).\qedhere
\end{equation*}
\popQED
\end{letterthm}

\vspace{-6pt}In general, $G/\Delta^+$ will not be pro-$p$. However, $G$ will always have an open normal subgroup $H$ containing $\Delta^+$ satisfying the conditions of Theorem A; and if the minimal prime $\mathfrak{M}$ of $kG$ satisfies the property that $\mathfrak{M}\cap k\Delta^+$ is prime, then it is also true that $\mathfrak{M}\cap kH$ is a minimal prime ideal of $kH$, and $(\mathfrak{M}\cap kH)kG = \mathfrak{M}$ (see Lemma \ref{lem: interpretation of G centralising e}). Then Theorem A gives us an isomorphism $\psi: kH/\mathfrak{M}\cap kH \to M_t\Big( k'[[H/\Delta^+]]\Big)$. Furthermore, we may always choose $H$ such that $H/\Delta^+$ is uniform, so that the prime spectrum of $k'[[H/\Delta^+]]$ is in theory well understood.

Unfortunately, even in the situation described in the above paragraph, the map $\psi$ may \emph{not} in general extend to an isomorphism $kG/\mathfrak{M} \to M_t\left( k'[[G/\Delta^+]]\right),$ as one might hope.

However, this guess is not too far off. Our next result shows that a similar statement does hold, provided we are willing to replace $k'[[G/\Delta^+]]$ by a closely related ring $\big(k'[[G/\Delta^+]]\big)_\alpha$. (Here, $\alpha$ is a 2-cocycle. The operation $(-)_\alpha$, which we call a \emph{central 2-cocycle twist}, deforms the ring $k'[[G/\Delta^+]]$; the resulting twisted ring can be expressed as a crossed product of $k'[[H/\Delta^+]]$ by the finite group $G/H$. We define $(-)_\alpha$ fully in Definition \ref{defn: cocycle twist}, and we then show that it preserves some desirable properties of the original ring.)

\begin{letterthm}[B]
Let $H$ be an open normal subgroup of $G$ containing $\Delta^+$, such that $H/\Delta^+$ is pro-$p$. Take a minimal prime $\mathfrak{M}$ of $kG$. Suppose that $\mathfrak{M}\cap k\Delta^+$ remains prime as an ideal of $k\Delta^+$. Then the isomorphism $\psi: kH/\mathfrak{M}\cap kH \to M_t\Big( k'[[H/\Delta^+]]\Big)$ of Theorem A extends to an isomorphism
\pushQED{\qed}
\begin{equation*}
\tilde{\psi}: kG/\mathfrak{M} \to M_t\Big( \big(k'[[G/\Delta^+]]\big)_\alpha\Big).\qedhere
\end{equation*}
\popQED
\end{letterthm}

\vspace{-6pt}Having understood $k'[[H/\Delta^+]]$, the study of the prime ideal structure of $(k'[[G/\Delta^+]])_\alpha$ now becomes amenable to finite crossed product techniques, as in \cite{passmanICP}.

Given an ideal $I$ of a completed group ring $kG$, it is often possible to associate to it a (closed normal) subgroup $H_0$ of $G$ in a way that sheds light on the nature of $I$. For example, $I$ may be the augmentation ideal of a closed normal subgroup $H_0$, or more generally $H_0$ may occur as the unfaithful part $I^\dagger$ of $I$, as defined in \cite[\S 1]{roseblade}; or else $H_0$ may contain the controller subgroup $I^\chi$ of $I$, as defined in \cite[\S 2.7]{ardakovContemp} (a careful refinement of the definition of \cite[\S 1]{roseblade}). These techniques were used in a crucial way by Roseblade in his treatise on prime ideals of polycyclic group algebras \cite{roseblade}, and play a similar role in the theory of Iwasawa algebras.

The construction of the maps $\psi$ and $\tilde\psi$ in Theorems A and B is very explicit, and these maps usually behave well with respect to closed normal subgroups. See \S\S 3.2--3.3 for details.

Studying the isomorphism of Theorem B in detail allows us to understand the behaviour of ideals of $kG$ containing $\mathfrak{M}$ by understanding ideals of $\big(k'[[G/\Delta^+]]\big)_\alpha$. We derive some consequences of Theorem B that will be useful in later work on extensions of prime ideals in such rings.

In Definition \ref{defn: dagger} below, we will say that an ideal $I$ of $kG$ is \emph{faithful} if the natural map $G\to (kG/I)^\times$ is an injection, and \emph{almost faithful} if its kernel is finite. A measure of the failure of $I$ to be faithful is given by the normal subgroup $I^\dagger := \ker(G\to (kG/I)^\times)$.

\begin{letterthm}[C]
With notation as in Theorem B, let $A$ be an ideal of $kH$ containing $\mathfrak{M}\cap kH$. Write $\psi(A/\mathfrak{M}\cap kH) = M_t(\mathfrak{a})$ for some ideal $\mathfrak{a}$ of $k'[[H/\Delta^+]]$. Then
\pushQED{\qed}
\begin{itemize}
\item[(i)] $A$ is prime in $kH$ if and only if $\mathfrak{a}$ is prime in $k'[[H/\Delta^+]]$.
\item[(ii)] $A$ is stable under conjugation by $G$ if and only if $\mathfrak{a}$ is stable under conjugation by $G/\Delta^+$ in the ring $\big(k'[[G/\Delta^+]]\big)_\alpha$.
\item[(iii)] $A$ is almost faithful as an ideal of $kH$ if and only if $\mathfrak{a}$ is (almost) faithful as an ideal of $k'[[H/\Delta^+]]$.\qedhere
\end{itemize}
\popQED
\end{letterthm}

We state all of these results together here for convenience. Statement (i) above is an easy consequence of Morita equivalence, but statements (ii) and (iii) rely crucially on explicit calculations under the isomorphism $\psi$.

In the general case, $\mathfrak{M}\cap k\Delta^+$ will \emph{not} necessarily be prime, but instead will be $G$-prime \cite[\S 14]{passmanICP}, and hence will equal the intersection of a (finite) $G$-orbit of primes of $k\Delta^+$ (cf. \cite[\S 5]{ardakovGMJ}). In this case, we have:

\begin{letterthm}[D]
Let $\mathfrak{M}$ be a minimal prime ideal of $kG$, and let $\mathfrak{p}$ be a minimal prime ideal of $k\Delta^+$ containing $\mathfrak{M}\cap k\Delta^+$. Let $G_1$ be the open subgroup of $G$ stabilising $\mathfrak{p}$. Then there exist an integer $r>1$, a minimal ideal $\mathfrak{N}$ of $kG_1$ with $\mathfrak{N}\cap k\Delta^+ = \mathfrak{p}$, and an isomorphism
\pushQED{\qed}
\begin{equation*}
\widetilde{\mu}: kG/\mathfrak{M} \to M_{r}\big(kG_1/\mathfrak{N}\big).\qedhere
\end{equation*}
\popQED
\end{letterthm}
\vspace{-10pt}
We prove a more precise statement of this theorem in subsection \ref{section: matrix units}, but do not state it here as the notation is rather technical. The more precise statement helps in understanding the relationship between ideals of $kG$ and ideals of $kH$, when $H$ is a closed normal subgroup of $G$. When $H$ acts transitively on the $G$-orbit of $\mathfrak{p}$, as in the special case of Theorem B, it is not hard to generalise Theorem C; but the $G$-orbit of $\mathfrak{p}$ may split into several $H$-orbits, and it is important to keep track of the isomorphism.

As an application of this, we finally prove the following result, which we will need in a future paper:

\begin{letterthm}[E]
Retain the notation of Theorem D. Let $P$ be an ideal of $kG$ containing $\mathfrak{M}$, and write $Q$ for the ideal of $kG_1$ such that $\widetilde{\mu}(P/\mathfrak{M}) = M_r(Q/\mathfrak{N}
)$. Let $H$ be any closed normal subgroup of $G$ containing $\Delta^+$, and $H_1 = H\cap G_1$. The following properties hold:
\pushQED{\qed}
\begin{itemize}
\item[(i)] $P$ is prime if and only if $Q$ is prime;
\item[(ii)] $P$ is controlled by $H$ if and only if $Q$ is controlled by $H_1$; and
\item[(iii)] $\displaystyle P^\dagger = \bigcap_{g\in G} \left(Q^\dagger\right)^g$.\qedhere
\end{itemize}
\popQED
\end{letterthm}

\vspace{-6pt}For the sake of clarity, much of the material in the main body of the paper leading up to the proofs of Theorems C and E is presented as a sequence of four separate ideal correspondences, labelled (I--IV). However, in order avoid obfuscating technicalities, we have chosen to omit mention of them in the statements of the theorems given in this introduction. We collate here, for the reader's convenience, references to the main results about these correspondences in the text.

$ $

\centerline{
\begin{tabular}{c|c|c c c}
Corresp.&Definition&Primality&Control&Faithfulness\\
\hline(I--III)&\ref{defn: corresps}&\ref{propn: corresps preserve primality}&\ref{propn: corresps preserve control}&\ref{propn: untwisted faithful primes remain faithful primes}\\
(IV)&\ref{defn: corresps 2}&\ref{propn: IV preserves primality}&\ref{propn: IV preserves control}&\ref{propn: IV preserves faithfulness}
\end{tabular}
}

\newpage
\section{Preliminaries}

\subsection{Ring-theoretic background}

\begin{defn}\cite[Introduction]{brumer}
Let $G$ be a profinite group and $k$ a commutative pseudocompact ring. Then the completed group ring $kG$ is defined to be
$$kG = \underset{N}{\underleftarrow{\lim}}\; k[G/N],$$
where the inverse limit ranges over all open normal subgroups $N$ of $G$, and $k[G/N]$ denotes the usual group ring of the (finite) group $G/N$ over $k$.
\end{defn}

\begin{rk}
When it helps to reduce ambiguity, we will write $kG$ as $k[[G]]$.
\end{rk}

Recall from \cite[1.5.2]{MR} that, for any finite group $F$, there is a natural embedding of groups $i_F: F\to (k[F])^\times$. By taking the inverse limit of the maps $\{i_{G/N}\}$, we get a continuous embedding $i: G\to (kG)^\times$.

\begin{lem}[Universal property of completed group rings]\label{defn: u.p. of group rings}
Let $G$ be a profinite group and $k$ a commutative pseudocompact ring. Then the completed group ring $kG$ satisfies the following universal property: given any pseudocompact $k$-algebra $R$ and any continuous group homomorphism $f: G\to R^\times$, there is a unique homomorphism $f^*:kG\to R$ of pseudocompact $k$-algebras satisfying $f^* \circ i = f$.
\end{lem}

\begin{proof}
Let $f: G\to R^\times$ be a continuous group homomorphism, and let $I$ be an open ideal of $R$ (hence a neighbourhood of zero) with $R/I$ artinian. Then $(I+1)\cap R^\times$ must be open in $R^\times$, and so its preimage $I^\dagger := f^{-1}((I+1)\cap R)$ must be open in $G$. Thus the map $f$ descends to a homomorphism of (abstract) groups $f_I: G/I^\dagger \to R^\times/(I+1)\cap R^\times \to (R/I)^\times$, and $G/I^\dagger$ is finite.

Now, by the universal property for (usual) group rings \cite[1.5.2]{MR}, we get a unique ring homomorphism $k[G/I^\dagger] \to R/I$ extending $f_I$, and hence a ring homomorphism $kG\to k[G/I^\dagger]\to R/I$ by definition. But as $R$ is the inverse limit of these $R/I$ \cite[Theorem 8.5(2)]{warner}, and the maps $f_I$ are all clearly compatible by uniqueness, we get a continuous ring homomorphism $kG\to R$ extending $f$.
\end{proof}

\begin{defn}[Universal property of completed tensor product]\label{defn: u.p. of tensor products} \cite[\S 2]{brumer}
Let $k$ be a commutative pseudocompact ring, and $R$ be a pseudocompact $k$-algebra. Let also $A$ be a right and $B$ a left pseudocompact $R$-module. Then the completed tensor product
$$A\hat{\tensor{R}} B$$
is a $k$-module satisfying the following universal property: there is a unique $R$-bihomo\-morphism
$$A\times B\to A\hat{\tensor{R}} B$$
through which any given $R$-bihomomorphism $A\times B\to C$ into a pseudocompact $k$-module $C$ factors uniquely. (An \emph{$R$-bihomomorphism} $\theta: A\times B\to C$ is a continuous $k$-module homomorphism satisfying $\theta(ar,b) = \theta(a,rb)$ for all $a\in A, b\in B, r\in R$.)

If $R = k$, and $A$ and $B$ are $k$-algebras, then their completed tensor product is also a $k$-algebra.
\end{defn}

We give also a construction.

\begin{lem} \cite[\S 2]{brumer}
Let $k$ be a commutative pseudocompact ring and $R$ a pseudocompact $k$-algebra. Let $A$ be a right and $B$ a left pseudocompact $R$-module: then the $k$-module defined by
$$A\hat{\tensor{R}} B := \underset{U,V}{\underleftarrow{\lim}}\;\Big( A/U \tensor{R} B/V\Big),$$
where $U$ and $V$ range over the open submodules of $A$ and $B$ respectively, satisfies the universal property for the completed tensor product of $A$ and $B$.\qed
\end{lem}

\subsection{Minimal prime ideals and centrally primitive idempotents $e$}

\textbf{Throughout this subsection}, $G$ is a compact $p$-adic analytic group and $k$ is a finite field of characteristic $p$.

Recall that the \emph{finite radical} of $G$ is defined as $$\Delta^+(G) = \{x\in G \,\big|\, o(x) < \infty, [G:\mathbf{C}_G(x)]<\infty\},$$
as in the text just before \cite[Lemma 4.1.6]{passmanASGR}. Recall also that an idempotent $e$ of a ring $R$ is \emph{centrally primitive} if $e$ is central in $R$ and $e$ cannot be written as a sum $e = f+f'$ of two nonzero orthogonal central idempotents.

We record here, for ease of reference, some facts and notation that we will use throughout this paper. The following result is due to Ardakov \cite{ardakovGMJ}.

\begin{lem}\label{lem: definition of e}
Write $J := J(k\Delta^+)$. Then $JkG$ is a two-sided ideal of $kG$ contained in the prime radical of $kG$. Denote by $\overline{(-)}$ images under the natural map $kG \to kG/JkG$.

\begin{itemize}
\item[(i)] There is a one-to-one correspondence

\centerline{
\xymatrix{
\left\{ \txt{minimal prime \\ ideals of $kG$} \right\}\ar@{<->}[r] &
\left\{ \txt{minimal prime \\ ideals of $\overline{kG}$} \right\}.
}
}

\item[(ii)] Let $X = \{e_1, \dots, e_r\}$ be a $G$-orbit of centrally primitive idempotents of $\overline{k\Delta^+}$, and write $f = e_1 + \dots + e_r$. Then $\overline{\mathfrak{M}_X} := (1-f)\overline{kG}$ is a minimal prime ideal of $\overline{kG}$, and its preimage $\mathfrak{M}_X$ in $kG$ is a minimal prime ideal in $kG$.

Conversely, let $\mathfrak{M}$ be a minimal prime ideal of $kG$. Then there exists a $G$-orbit $X$ of centrally primitive idempotents of $\overline{k\Delta^+}$ such that $\mathfrak{M} = \mathfrak{M}_X$.

This sets up a one-to-one correspondence

\centerline{
\xymatrix{
\left\{ \txt{minimal prime \\ ideals of $kG$} \right\}\ar@{<->}[r] &
\left\{ \txt{$G$-orbits of centrally \\ primitive idempotents of $\overline{k\Delta^+}$} \right\}.
}
}

\item[(iii)] Given a centrally primitive idempotent $e\in \overline{k\Delta^+}$, there exist some $t>0$ and some finite field extension $k'/k$ with $e\cdot\overline{k\Delta^+} \cong M_t(k')$. Hence, if $A$ is a $k$-algebra, we may identify the rings

\centerline{
\xymatrix{
e\cdot\overline{k\Delta^+}\tensor{k} A\ar@{<.>}[r]^-=&e\cdot\overline{k\Delta^+}\tensor{k'} (k' \tensor{k} A)\ar@{<.>}[r]^-=&M_t(k'\tensor{k} A).
}
}
\end{itemize}
\end{lem}

\begin{proof}
$ $

\begin{itemize}
\item[(i)] This follows from \cite[5.2]{ardakovGMJ}.
\item[(ii)] This follows from \cite[\S 5, in particular 5.7]{ardakovGMJ} and part (i).
\item[(iii)] $e\cdot\overline{k\Delta^+}$ is a simple finite-dimensional $k$-algebra, so the isomorphism $e\cdot\overline{k\Delta^+} \cong M_t(k')$ follows from Wedderburn's theorem. The rest is a simple calculation.\qedhere
\end{itemize}
\end{proof}

\begin{rk}\label{rk: f-M correspondence}
Retain the notation of Lemma \ref{lem: definition of e}(ii), and take a minimal prime $\mathfrak{M} = \mathfrak{M}_X$ of $kG$, so that $\overline{\mathfrak{M}} = (1-f)\overline{kG}$. As $f$ is idempotent, this implies that $f\cdot\overline{kG} = \overline{kG} / \overline{\mathfrak{M}} \cong kG/\mathfrak{M}$; also, if $I$ is an ideal of $kG$ containing $\mathfrak{M}$, then $f\cdot\overline{I} = \overline{I} / \overline{\mathfrak{M}} \cong I/\mathfrak{M}$.
\end{rk}

\begin{lem}\label{lem: unique minimal primes}
If a proper ideal $I\lhd kG$ contains a minimal prime ideal, then it contains a \emph{unique} minimal prime ideal.
\end{lem}

\begin{proof}
Suppose $I$ contains both $\mathfrak{M}_X$ and $\mathfrak{M}_{X'}$ (in the notation of Lemma \ref{lem: definition of e}(ii)), where $X$ and $X'$ are two distinct $G$-orbits of centrally primitive idempotents of $\overline{k\Delta^+}$. Continuing to denote images under the map $kG\to kG/J(k\Delta^+)kG$ by $\overline{(-)}$, we see that $\overline{I}$ contains both $\overline{\mathfrak{M}_X}$ and $\overline{\mathfrak{M}_{X'}}$.

Suppose $X = \{e_1, \dots, e_r\}$ and $X' = \{e'_1, \dots, e'_s\}$. Write $f = e_1 + \dots + e_r$ and $f' = e'_1 + \dots + e'_s$. Then $\overline{\mathfrak{M}_X} = (1-f)\overline{kG}$ and $\overline{\mathfrak{M}_{X'}} = (1-f')\overline{kG}$. In particular, $\overline{I}$ contains both $1-f$ and $1-f'$; so it contains their product, which is $1-f-f'$; and hence it contains $(1-f) + (1-f') - (1-f-f') = 1$, and so $I = R$.
\end{proof}

We will use this correspondence very often while working in the following setup, so we give it a name:

\begin{setup}\label{setup: k, G, bar}
Let $k$ be a finite field of characteristic $p$, $G$ a compact $p$-adic analytic group, $\Delta^+$ its finite radical, and $\overline{(-)}: kG\to kG/J(k\Delta^+)kG$ the natural map. Here, Lemma \ref{lem: definition of e}, Remark \ref{rk: f-M correspondence} and Lemma \ref{lem: unique minimal primes} apply, as well as Lemma \ref{lem: interpretation of G centralising e} below, and we will sometimes use these results without special mention whenever this Setup is invoked.
\end{setup}

We also define some useful notation.

\begin{notn}\label{notn: notation for e}
Assume the hypotheses of Setup \ref{setup: k, G, bar}.

Suppose the proper ideal $I\lhd kG$ contains a minimal prime ideal, and hence (by Lemma \ref{lem: unique minimal primes}) contains a \emph{unique} minimal prime ideal, say $\mathfrak{M}$. Then, under the correspondence of Lemma \ref{lem: definition of e}(ii), we obtain a unique $G$-orbit $X$ of centrally primitive idempotents of $\overline{k\Delta^+}$ corresponding to $\mathfrak{M}$.

Throughout this paper, we will write $$\cpi(I)$$ for this set $X$. Given a centrally primitive idempotent $e\in\cpi(I)$, we will write $f = \cpisum{e}{G}$ to mean $$f = \sum_{g\in \mathbf{C}_G(e)\backslash G} e^g,$$ where $\mathbf{C}_G(e)\backslash G$ denotes the (finite) set of right cosets of $\mathbf{C}_G(e)$ in $G$. In other words, if we write $e = e_1$ and $X = \{e_1, \dots, e_r\}$, then $f = \cpisum{e}{G}$ means $f = e_1 + \dots + e_r$.
\end{notn}

\begin{rk}\label{rk: assumption on G, e}
Let $\mathfrak{M}$ be a minimal prime of $kG$, and $e\in\cpi(\mathfrak{M})$ (as in Notation \ref{notn: notation for e}). For most of the rest of this paper, we will insist on the mild condition that $G$ centralise $e$. This is mostly to keep the notation simple: we will return briefly to this issue in \S \ref{section: matrix units}, and show that we have not lost much generality by doing this.

We note that this condition on $G$ and $e$ has a natural interpretation in terms of extension and contraction of primes:
\end{rk}

\begin{lem}\label{lem: interpretation of G centralising e}
Assume the hypotheses of Setup \ref{setup: k, G, bar}. Let $\mathfrak{M}$ be a minimal prime of $kG$, fix $e\in\cpi(\mathfrak{M})$, and assume that $G$ centralises $e$. Then, given any closed normal subgroup $H$ of $G$ containing $\Delta^+$, we have that $\mathfrak{M}\cap kH$ is a minimal prime ideal of $kH$, and $(\mathfrak{M}\cap kH)kG = \mathfrak{M}$.
\end{lem}

\begin{proof}
This follows trivially from the correspondence of Lemma \ref{lem: definition of e}(i), (ii), and Remark \ref{rk: f-M correspondence}.
\end{proof}

\subsection{Idempotents and the Peirce decomposition}\label{subsection: peirce}

\begin{defn}\label{defn: peirce}
Let $R$ be a ring containing mutually orthogonal nonzero idempotents $e_1, \dots, e_r\in R$ whose sum is $1$. Recall the \emph{Peirce decomposition} of $R$ with respect to this set of idempotents,
\begin{equation}\label{eqn: peirce}
R = \sum_{i,j=1}^r R_{ij},
\end{equation}
a sum of additive subgroups, where $R_{ij} := e_i R e_j$.
\end{defn}

In this subsection, we prove a basic property of the Peirce decomposition. We are certain that it is well-known, but were unable to find a reference for it.

For brevity, we will say that $A\subseteq R$ is an \emph{$\mathbf{e}$-bimodule} if it is an additive subgroup of $R$ which is closed under left and right multiplication by the idempotent $e_i$ for each $1\leq i\leq r$; and by its $(i,j)$-component we will mean the (sub-)$\mathbf{e}$-bimodule $e_i A e_j$. (Arbitrary ideals $I\lhd R$, and subrings $R'\leq R$ which contain all of the $e_i$ for $1\leq i\leq r$, are $\mathbf{e}$-bimodules.)

\begin{lem}\label{lem: peirce and e-bimods}
Let $A\subseteq R$ be an $\mathbf{e}$-bimodule, and write its $(i,j)$-component as $A_{ij}$. Then
\begin{itemize}
\item[(i)] $A = \sum_{i,j=1}^r A_{ij}$, and
\item[(ii)] $A_{ij}\cap A_{kl} = \{0\}$ for $(i,j) \neq (k,l)$.
\end{itemize}
In particular, (\ref{eqn: peirce}) is a \emph{direct} sum, as is the decomposition of $A$ in part (i) for any $\mathbf{e}$-bimodule $A$.
\end{lem}

\begin{proof}
By assumption, $A_{ij}\subseteq A$ for all $1\leq i,j\leq r$, and so
$$\sum_{i,j=1}^r A_{ij} \subseteq A.$$
Conversely, we always have
\begin{align*}
A = 1\cdot A\cdot 1 &= (e_1 + \dots + e_r) A (e_1 + \dots + e_r)\\
&\subseteq \sum_{i=1}^r \sum_{j=1}^r e_i A e_j = \sum_{i,j=1}^r A_{ij}.
\end{align*}
This establishes (i). For (ii): suppose that $e_i a e_j = e_k b e_l\in A_{ij}\cap A_{kl}$ for some $(i,j) \neq (k,l)$ and $a, b\in A$. As $(i,j)\neq (k,l)$, we have either $e_i e_k = 0$ or $e_l e_j = 0$ (or both), and so, premultiplying by $e_i$ and postmultiplying by $e_j$, we get that $e_i a e_j = 0$, which is what we wanted to show.
\end{proof}

\begin{rk}\label{rk: peirce decomp is a matrix ring}
We may introduce here the suggestive ``matrix" notation
$$R = \bigoplus_{i,j=1}^r R_{ij} =
\begin{pmatrix}
R_{11} & R_{12} & \dots & R_{1r} \\
R_{21} & R_{22} & \dots & R_{2r} \\
\vdots & \vdots & \ddots & \vdots \\
R_{r1} & R_{r2} & \dots & R_{rr}
\end{pmatrix}.$$
We will not use this notation in the rest of the paper, but it forms the intuition underlying the Peirce decomposition and some results of our \S 5: $R$ may be viewed as a generalised ``matrix ring", with entries in these additive groups $R_{ij}$. To make this precise: each $R_{ii}$ is naturally a ring with identity $1_{R_{ii}} = e_i$ under the multiplication inherited from $R$; each $R_{ij}$ is an $(R_{ii}, R_{jj})$-bimodule; and the restriction of the multiplication map on $R$, which we may view as the $(R, R)$-bimodule map $R\tensor{R_{jj}} R\to R$, gives a homomorphism of $(R_{ii}, R_{kk})$-bimodules $R_{ij}\tensor{R_{jj}} R_{jk} \to R_{ik}$ for all $i, j, k$. (Compare Lemma \ref{lem: MUL}.)
\end{rk}

\section{The untwisting theorem}\label{section: untwisting theorem}

\subsection{An untwisting map}

\begin{thm}[untwisting]\label{thm: untwisting map delta}
Let $G$ be a compact $p$-adic analytic group, $k$ a commutative pseudocompact ring, and $kG$ the associated completed group ring. Suppose $H$ is a closed normal subgroup of $G$, and $I$ is an ideal of $kH$ such that $I\cdot kG = kG\cdot I$. Write $\pi: kG\to kG/IkG$, so that $\pi(kH) = kH/I$. Suppose also that we have a continuous group homomorphism $\delta: G\to \pi(kH)^\times$ satisfying
\begin{itemize}
\item[(i)] $\delta(g) = \pi(g)$ for all $g\in H$,
\item[(ii)] $\delta(g)^{-1} \pi(g)$ centralises $\pi(kH)$ for all $g\in G$.
\end{itemize}
Then there exists an isomorphism of pseudocompact $k$-algebras
$$\Psi: \pi(kG) \to \pi(kH) \hat{\tensor{k}} k[[G/H]],$$
where $\hat\otimes$ denotes the completed tensor product.
\end{thm}

\begin{proof}
For ease of notation, we will write $\pi$ as $\overline{(-)}$ throughout the proof.

Firstly, the function $G\to \left(\overline{kH}\hat{\tensor{k}} k[[G/H]]\right)^\times$ given by
$$g\mapsto \delta(g) \otimes gH$$
is a continuous group homomorphism, and so the universal property of completed group rings of Lemma \ref{defn: u.p. of group rings} allows us to extend this function uniquely to a continuous ring homomorphism
$$\Psi': kG \to \overline{kH}\hat{\tensor{k}} k[[G/H]].$$
In the same way, we may extend $\delta$ uniquely to a map $\delta: kG\to \overline{kH}$, and by assumption (i), $\delta|_{kH}$ is just the natural quotient map $kH\to \overline{kH}$. Hence $\ker \delta$ must contain the two-sided ideal $IkG$, so that $\Psi'$ descends to a continuous ring homomorphism
$$\Psi: \overline{kG} \to \overline{kH}\hat{\tensor{k}} k[[G/H]].$$
We claim that this is the desired isomorphism. To show that $\Psi$ is an isomorphism, we will construct a continuous ring homomorphism
$$\Phi: \overline{kH}\hat{\tensor{k}} k[[G/H]]\to \overline{kG}$$
and show that $\Phi$ and $\Psi$ are mutually inverse.

Consider the continuous function $\varepsilon: G\to \overline{kG}^\times$ given by $$\varepsilon(g) = \delta(g)^{-1} \overline{g}.$$ $\varepsilon$ is a group homomorphism: indeed, for all $g, h\in G$, we have
\begin{align*}
\varepsilon(g)\varepsilon(h) &= \delta(g)^{-1} \overline{g}\, \delta(h)^{-1} \overline{h}\\
&= \boxed{\delta(g)^{-1} \overline{g}}\, \boxed{\delta(h)^{-1}}\, \overline{h}\\
&= \boxed{\delta(h)^{-1}}\, \boxed{\delta(g)^{-1} \overline{g}}\, \overline{h} & \text{by assumption (ii)}\\
&= \delta(gh)^{-1} \overline{gh} = \varepsilon(gh).
\end{align*}
By assumption (i), $\ker \varepsilon$ contains $H$, and so $\varepsilon$ descends to a continuous group homomorphism $\varepsilon: G/H \to \overline{kG}^\times$; and so again by the universal property we get a continuous ring homomorphism $\varepsilon': k[[G/H]] \to \overline{kG}$. We also clearly have a continuous inclusion $\overline{kH} \to \overline{kG}$.

These functions, and the universal property of completed tensor products of Definition \ref{defn: u.p. of tensor products}, allow us to define the desired map $\Phi: \overline{kH}\hat{\tensor{k}} k[[G/H]]\to \overline{kG}$ by
$$\Phi(x\otimes y) = x\varepsilon'(y).$$
This map $\Phi$ is bilinear in its arguments by construction; to show that it is a ring homomorphism, we need only show that
$$\Phi(x_1\otimes y_1)\Phi(x_2\otimes y_2) = \Phi(x_1x_2\otimes y_1y_2),$$
i.e. that $\varepsilon'(y_1)$ commutes with $x_2$ inside $\overline{kG}$: but this is assumption (ii).

It now remains only to check that $\Phi$ and $\Psi$ are mutually inverse. Indeed, for all $g\in G$ and $x\in \overline{kH}$,
\begin{align*}
\Phi(\Psi(\overline{g})) & = \Phi(\delta(g) \otimes gH) \\
& = \Phi(\delta(g)\otimes H) \Phi(1 \otimes gH) \\
& = (\delta(g)) (\varepsilon(g)) \\
& = \overline{g},
\end{align*}
and
\begin{align*}
\Psi(\Phi(x\otimes gH)) & = \Psi(x\varepsilon(g)) \\
& = \Psi(x\delta(g)^{-1})\Psi(\overline{g}) \\
& = (x\delta(g)^{-1}\otimes 1)(\delta(g)\otimes gH) \\
& = x\otimes gH.\qedhere
\end{align*}
\end{proof}

We now apply this to our context, for simplicity assuming that $G$ centralises $e$ throughout (see Remark \ref{rk: assumption on G, e}).

\begin{cor}\label{cor: untwisting map delta}
Assume Setup \ref{setup: k, G, bar}. Take $\mathfrak{M}$ a minimal prime of $kG$ and $e\in\cpi(\mathfrak{M})$, and assume that $G$ centralises $e$. Suppose further that we are given a continuous group homomorphism
$$\delta: G\to (e\cdot\overline{k\Delta^+})^\times,$$
satisfying

\begin{itemize}
\item[(i)] $\delta(g) = e\cdot\overline{g}$ for all $g\in\Delta^+$,
\item[(ii)] $\delta(g)^{-1} \overline{g}$ centralises $e\cdot\overline{k\Delta^+}$ for all $g\in G$.
\end{itemize}

Then there exists an isomorphism
$$\Psi: e\cdot\overline{kG} \to e\cdot\overline{k\Delta^+}\tensor{k} k[[G/\Delta^+]],$$
and hence also an isomorphism
$$\psi: e\cdot\overline{kG}\to M_t(k'[[G/\Delta^+]])$$
for some positive integer $t$ and some finite field extension $k'/k$.
\end{cor}

\begin{proof}
In the notation of the statement of Theorem \ref{thm: untwisting map delta}, take $H = \Delta^+$, and take $I$ to be the ideal of $k\Delta^+$ generated by $J(k\Delta^+)$ and $1-a$, where $a\in k\Delta^+$ is any element whose image in $\overline{k\Delta^+}$ is $e$. Hence $\pi: kG\to kG/IkG$ is precisely the natural map $kG\to e\cdot\overline{kG}$. Now, Theorem \ref{thm: untwisting map delta} gives the isomorphism $$\Psi: e\cdot\overline{kG} \to e\cdot\overline{k\Delta^+}\hat{\tensor{k}} k[[G/\Delta^+]];$$ but now, as $e\cdot\overline{k\Delta^+}$ is finite-dimensional as a vector space over $k$, \cite[Lemma 2.1(ii)]{brumer} implies that the right hand side is equal to the ordinary tensor product $$e\cdot\overline{k\Delta^+}\tensor{k} k[[G/\Delta^+]].$$ Now the isomorphism $\psi$ is given by composing $\Psi$ with the isomorphism of Lemma \ref{lem: definition of e}(iii).
\end{proof}

In order to use this result, we will need to find such a $\delta$ for a certain large class of groups $G$. We postpone doing so until \S \ref{section: finite-by-pro-p}.

\begin{rk}
Once we have found an untwisting map $\delta$ as above, this result will be a strong generalisation of the result given in \cite[10.1]{ardakovJAG}, in the case when $G \cong F\times U$, for $F$ a finite $p'$-group and $U$ a uniform group. In this case, $\Delta^+(G) = F$, and we may simply take $\delta$ to be the composite of the natural projection map $F\times U\to F = \Delta^+$ with the inclusion map $\Delta^+\hookrightarrow (e\cdot\overline{k\Delta^+})^\times$.
\end{rk}

%

\section{Properties of ideals}

\subsection{Three ideal correspondences}

At this stage, it will be useful to mention explicitly some one-to-one correspondences of ideals that we are developing.

In the below definition, we will use the following notation. Given a ring $R$ with an ideal $J$, we will write the lattice of all (two-sided) ideals of $R$ as $\mathcal{I}(R)$, and the lattice of all ideals of $R$ which contain $J$ as $\mathcal{I}_J(R)$.

\begin{defn}\label{defn: corresps} Adopt Setup \ref{setup: k, G, bar}, and let $\mathfrak{M}$ be a minimal prime ideal of $kG$ and $e\in\cpi(\mathfrak{M})$. Set $f = \cpisum{e}{G}$. We define three ideal correspondences (I), (II), (III) as follows.

\begin{itemize}
\item[(I)]
We always have the ideal correspondence
\begin{equation*}
\begin{minipage}{0.9\textwidth}
\xymatrix{
\mathcal{I}_{\mathfrak{M}}(kG)\ar@{<->}[r]&
\mathcal{I}(f\cdot\overline{kG})
}
\end{minipage}
\end{equation*}
induced by the natural quotient map $q: kG\to f\cdot\overline{kG}$ sending $I$ to $f\cdot\overline{I}$.
\item[(II)]
Suppose that $e$ is $G$-invariant, so that $f = e$, and that we have an untwisting map $\delta: G\to (e\cdot\overline{k\Delta^+})^\times$ satisfying the hypotheses of Corollary \ref{cor: untwisting map delta}. Let $e\cdot\overline{k\Delta^+} \cong M_t(k')$ for some integer $t\geq 1$ and some finite field extension $k'/k$ as in Lemma \ref{lem: definition of e}(iii). Then we have the ideal correspondence
\begin{equation*}
\begin{minipage}{0.9\textwidth}
\xymatrix{
\mathcal{I}(e\cdot\overline{kG}) \ar@{<->}[r]&
\mathcal{I}\left(M_t(k'[[G/\Delta^+]])\right)
}
\end{minipage}
\end{equation*}
induced by the map $\psi$ of Corollary \ref{cor: untwisting map delta}, sending $e\cdot\overline{I}$ to $\psi(e\cdot\overline{I}) = M_t(\mathfrak{i})$ for some ideal $\mathfrak{i}\lhd k'[[G/\Delta^+]]$.
\item[(III)]
We always have the ideal correspondence
\begin{equation*}
\begin{minipage}{0.9\textwidth}
\xymatrix{
\mathcal{I}\left(M_t(k'[[G/\Delta^+]])\right) \ar@{<->}[r]&
\mathcal{I}\left(k'[[G/\Delta^+]]\right)
}
\end{minipage}
\end{equation*}
given by Morita equivalence, sending $M_t(\mathfrak{i})$ to $\mathfrak{i}$: see \cite[3.5.5]{MR}.
\end{itemize}
\end{defn}

The remainder of this section is dedicated to showing that these three correspondences preserve certain desirable properties of ideals.

\subsection{Inclusion and primality}

\begin{propn}\label{propn: corresps preserve primality}
The correspondences (I--III) of Definition \ref{defn: corresps} are order-preserving maps of posets, and send prime ideals to prime ideals.
\end{propn}

\begin{proof}
$ $

\begin{itemize}
\item[(I)] Recall from Remark \ref{rk: f-M correspondence} that $f\cdot\overline{kG} \cong kG/\mathfrak{M}$. Now the claim reduces to standard ``isomorphism theorems".
\item[(II)] $\psi$ is an isomorphism.
\item[(III)] This follows from \cite[3.5.5, 3.5.9]{MR}.\qedhere
\end{itemize}
\end{proof}

\subsection{Control by subgroups}

Recall the following standard definition.

\begin{defn}\label{defn: control}
Let $I$ be an ideal of the ring $R$. We say that $I$ is \emph{controlled by} the subring $S$ if $(I\cap S)R = I$.

Following Roseblade \cite[1.1]{roseblade}: let $k$ be a commutative pseudocompact ring, $G$ a compact $p$-adic analytic group and $H$ a closed normal subgroup of $G$. If the ideal $I\lhd kG$ is controlled by the subring $kH$, we say $I$ is \emph{controlled by} the subgroup $H$.
\end{defn}

In other words, to understand ideals of $R$ which are controlled by $S$, we need only understand ideals of $S$. This is helpful, as in many cases of interest $S$ can often be taken to be a smaller ring, or one whose ideals are better understood: this was crucial in Roseblade's work \cite{roseblade} on the prime ideals of group algebras of polycyclic groups. One recent important result in the area of Iwasawa algebras along these lines includes \cite[Theorem A(b)]{ardakovInv}.

It is natural to ask how the notion of ``control" is preserved under correspondences (I--III). The next two lemmas answer this question.

\begin{lem}\label{lem: reduction to quotient of kG}
Adopt Setup \ref{setup: k, G, bar}. Let $I$ be an ideal of $kG$ containing a prime ideal, and write $e = \cpi(I), f = \cpisum{e}{G}$ as in Notation \ref{notn: notation for e}. Let $H$ be any closed normal subgroup of $G$ containing $\Delta^+$. Then the following are equivalent:

\begin{itemize}
\item[(i)] $(I\cap kH)kG = I$,
\item[(ii)] $(\overline{I} \cap \overline{kH}) \overline{kG} = \overline{I}$,
\item[(iii)] $(f\cdot\overline{I} \cap f\cdot\overline{kH}) f\cdot\overline{kG} = f\cdot\overline{I}$.
\end{itemize}
\end{lem}

\begin{proof}
Write $A := J(k\Delta^+)kG$, and write $B$ for the full preimage in $kG$ of $(1-f)\overline{kG}$. Then $A\subseteq B\subseteq I\cap k\Delta^+$ by Lemma \ref{lem: definition of e}.

$\boxed{\text{(i)} \Leftrightarrow \text{(ii)}}$
By the modular law, $(I\cap kH) + A = I\cap (kH + A)$, from which we can deduce that
$\overline{I\cap kH} = \overline{I} \cap \overline{kH},$
and so
$\overline{(I\cap kH)kG} = (\overline{I} \cap \overline{kH})\overline{kG}.$

Now, suppose that $(I\cap kH)kG = I$. Then the above argument shows that we may quotient out by $A$ on both sides to get $(\overline{I} \cap \overline{kH})\overline{kG} = \overline{I}$.

Conversely, suppose that $(\overline{I} \cap \overline{kH}) \overline{kG} = \overline{I}$. Then, taking preimages in $kG$ of both sides, we get $(I\cap kH)kG + A = I + A$. But both $I$ and $(I\cap kH)kG$ already contain $A$, so we can deduce that $(I\cap kH)kG = I$.

$\boxed{\text{(i)} \Leftrightarrow \text{(iii)}}$
The proof is identical to the previous case on replacing $A$ by $B$.
\end{proof}

\begin{lem}\label{lem: matrix-control}
Let $R$ be a ring, $I$ an ideal of $R$, and $S$ a subring of $R$. Let $t$ be a positive integer. Then the following are equivalent:
\begin{itemize}
\item[(i)] $(I\cap S)R = I$,
\item[(ii)] $(M_t(I) \cap M_t(S))M_t(R) = M_t(I)$.
\end{itemize}
\end{lem}

\begin{proof}
Write $d: R\to M_t(R)$ for the natural diagonal map $r\mapsto rI$, where $I$ is the identity element of $M_t(R)$. First, we will show that $(M_t(I) \cap M_t(S))M_t(R) = M_t((I\cap S)R)$.

It is immediate that $M_t(I)\cap M_t(S) = M_t(I\cap S)$.

Next, in order to show that $M_t((I\cap S)R) = M_t(I\cap S)d(R)$, we note simply that, for any $r\in R$ and elementary matrix $E_{ij}$ as above, we have $rE_{ij} = E_{ij}d(r)$. So, given some $x\in M_t((I\cap S)R)$, we may write
$$x = \sum_{i,j=1}^t x_{ij} E_{ij},$$
with each $x_{ij} \in (I\cap S)R$, so that
$$x_{ij} = \sum_{k=1}^{n_{ij}} y_{ijk} r_{ijk},$$
with each $y_{ijk}\in I\cap S$, each $r_{ijk}\in R$ and each $n_{ij}$ some positive integer. Now, reordering the factors of each product as above, we see that
$$x = \sum_{i,j,k} (y_{ijk} E_{ij}) d(r_{ijk})$$
is an element of $M_t(I\cap S)d(R)$; the converse is similar.

Finally, the inclusion $M_t(I\cap S)d(R) \subseteq M_t(I\cap S)M_t(R)$ is trivial; and the reverse inclusion is proved using a similar trick to the above, i.e. any element $x\in M_t(R)$ can be written as
$$x = \sum_{i,j=1}^t z_{ij} E_{ij} = \sum_{i,j=1}^t E_{ij} d(z_{ij})$$
for some $z_{ij}\in R$, and $M_t(I\cap S)E_{ij} \subseteq M_t(I\cap S)$. This establishes the equality.

In the following, we identify $R$, $S$, $I$, etc. with their images under $d$.

$\boxed{\text{(i)} \Rightarrow \text{(ii)}}$ It is straightforward to check that
$$I\cdot M_t(R) \subseteq M_t(I),$$
and conversely if $a_{kl} \in I$ for all $1 \leq k, l \leq t$ then
$$(a_{ij})_{i,j} = \sum_{k,l} a_{kl} E_{kl}$$
where $E_{kl}\in M_t(R)$ is the elementary matrix with $(i,j)$-entry $\delta_{ik} \delta_{jl}$. So we see that $I\cdot M_t(R) = M_t(I)$. Now
\begin{align*}
M_t(I) = I\cdot M_t(R) & = (I\cap S)R\cdot  M_t(R) & \text{by assumption (i)} \\
&= (I\cap S)M_t(R)\\
&\subseteq (M_t(I) \cap M_t(S)) M_t(R)\\
&= M_t((I\cap S)R) & \text{by the above}\\
&= M_t(I)& \text{by assumption (i),}
\end{align*}
and so these are all equal.

$\boxed{\text{(ii)} \Rightarrow \text{(i)}}$  By intersecting both sides of the equality $M_t((I\cap S)R) = M_t(I)$ with the subring $d(R)$, which we naturally identify with $R$, we see that assumption (ii) implies $(I\cap S)R = I$.
\end{proof}

In summary:

\begin{propn}\label{propn: corresps preserve control}
Adopt Setup \ref{setup: k, G, bar} and the notation of Definition \ref{defn: corresps}, and recall the correspondences (I--III) there. Let $H$ be a closed normal subgroup of $G$ containing $k\Delta^+$.
\begin{itemize}
\item[(I)] Let $I$ and $f\cdot\overline{I}$ be corresponding ideals as in Definition \ref{defn: corresps}(I). Then $I$ is controlled by $kH$ if and only if $f\cdot\overline{I}$ is controlled by $f\cdot\overline{kH}$.
\item[(II)] Assume that $G$ centralises $e$ and we have a map $\delta:G\to (e\cdot\overline{k\Delta^+})^\times$ satisfying the hypotheses of Corollary \ref{cor: untwisting map delta}. Let $e\cdot\overline{I}$ and $M_t(\mathfrak{i})$ be corresponding ideals as in Definition \ref{defn: corresps}(II). Then $e\cdot\overline{I}$ is controlled by $e\cdot\overline{kH}$ if and only if $M_t(\mathfrak{i})$ is controlled by $M_t(k'[[H/\Delta^+]])$.
\item[(III)] Let $M_t(\mathfrak{i})$ and $\mathfrak{i}$ be corresponding ideals as in Definition \ref{defn: corresps}(III). Then $M_t(\mathfrak{i})$ is controlled by $M_t(k'[[H/\Delta^+]])$ if and only if $\mathfrak{i}$ is controlled by $k'[[H/\Delta^+]]$.
\end{itemize}
\end{propn}

\begin{proof}
(I) follows from Lemma \ref{lem: reduction to quotient of kG}. (II) follows from the easy calculation that $\psi(e\cdot\overline{kH}) = M_t(k'[[H/\Delta^+]])$, as $\psi$ is a ring isomorphism. (III) follows from Lemma \ref{lem: matrix-control}.
\end{proof}

\subsection{Faithfulness}

\begin{defn}\label{defn: dagger}
Let $k$ be a commutative pseudocompact ring and $G$ a compact $p$-adic analytic group. Again following Roseblade \cite{roseblade}, for any ideal $I$ of $kG$, we define $$I^\dagger = \{x\in G \,\big|\, x-1 \in I\}.$$ This is the kernel of the natural continuous group homomorphism $G \to (kG/I)^\times$, and so is a closed normal subgroup of $G$. If $I^\dagger = 1$, we say that $I$ is a \emph{faithful} ideal. It also often occurs that $I^\dagger$ is finite, i.e. $I^\dagger\leq \Delta^+(G)$; in this case, we say that $I$ is \emph{almost faithful}.
\end{defn}

\begin{rk}
Note that the image of $I$ under the map $kG\to k[[G/I^\dagger]]$ will be a faithful ideal. Hence, when trying to understand ideals of $kG$ for a certain class of groups $G$ closed under quotients by closed normal subgroups of the form $I^\dagger$, we may usually reduce the problem to one of \emph{faithful} ideals.
\end{rk}

In this subsection, we show that some appropriate notion of (almost-)faithfulness is also preserved by the correspondences of Definition \ref{defn: corresps}.

\textbf{Throughout this subsection, adopt Setup \ref{setup: k, G, bar}, and fix the following notation throughout.}
\begin{itemize}
\item Let $\mathfrak{M}$ be a minimal prime ideal of $kG$, and fix $e\in\cpi(\mathfrak{M})$.
\item Suppose that $G$ centralises $e$, and write $q: kG\to e\cdot\overline{kG}$ for the natural quotient map.
\item Let $e\cdot\overline{k\Delta^+} \cong M_t(k')$ for some integer $t>1$ and some finite field extension $k'/k$ as in Lemma \ref{lem: definition of e}(iii).
\item Suppose that we have an untwisting map $\delta: G\to (e\cdot\overline{k\Delta^+})^\times$ satisfying the hypotheses of Corollary \ref{cor: untwisting map delta}, so that the correspondences (I--III) of Definition \ref{defn: corresps} hold. Continue to write $$\Psi: e\cdot\overline{kG} \to e\cdot\overline{k\Delta^+}\tensor{k} k[[G/\Delta^+]]$$ and $$\psi: e\cdot\overline{kG}\to M_t(k'[[G/\Delta^+]])$$ for the isomorphisms defined there.
\end{itemize}

Now let $A$ be an arbitrary ideal of $kG$ containing $\mathfrak{M}$. Then, by correspondences (I--III) of Definition \ref{defn: corresps}, we have
$$\psi\circ q(A) = M_t(\mathfrak{a}),$$
where $\mathfrak{a}$ is some ideal of $k'[[G/\Delta^+]]$. We intend to show that the groups $A^\dagger$ (a closed normal subgroup of $G$) and $\mathfrak{a}^\dagger$ (a closed normal subgroup of $G/\Delta^+$) are closely related.

$e\cdot\overline{k\Delta^+}$ is a $k'$-algebra, so let us abuse notation to identify the two rings $$e\cdot\overline{k\Delta^+}\tensor{k} k[[G/\Delta^+]] = e\cdot\overline{k\Delta^+}\tensor{k'} k'[[G/\Delta^+]]$$ in the obvious way. Then, laying out the structure more explicitly, we have
$$\Psi(e\cdot\overline{A}) = e\cdot\overline{k\Delta^+}\tensor{k'} \mathfrak{a}$$
in the notation of Corollary \ref{cor: untwisting map delta}.

Suppose that $g\Delta^+ \in \mathfrak{a}^\dagger$ for some $g\in G$, i.e. $g\Delta^+ - 1_G\Delta^+ \in \mathfrak{a}$. Then $$(1\otimes g\Delta^+) - (1\otimes \Delta^+) \in \Psi(e\cdot\overline{A}),$$ so $$\Phi(1\otimes g\Delta^+) - \Phi(1\otimes \Delta^+) = \delta(g)^{-1}\overline{g} - e \in e\cdot\overline{A}.$$ This motivates the following definition:

\begin{defn}\label{defn: delta-dagger}
Write $$A_\delta^\dagger = \{g\in G \,\big|\, \delta(g)^{-1} \overline{g} - e \in e\cdot\overline{A}\},$$ so that $A_\delta^\dagger$ is the kernel of the composite map 

\centerline{
\xymatrix{
G\ar[r]^-\varepsilon& (e\cdot\overline{kG})^\times\ar[r]^\sim & (kG/\mathfrak{M})^\times\ar@{->>}[r]& (kG/A)^\times,
}
}

where $\varepsilon: G\to (e\cdot\overline{kG})^\times$ is defined by $\varepsilon(g) = \delta(g)^{-1}\overline{g}$. (As we saw in the proof of Theorem \ref{thm: untwisting map delta}, $\varepsilon$ is a continuous group homomorphism.) Compare this with Definition \ref{defn: dagger}: this is a ``twisted" version of $A^\dagger$.

Now, since $\varepsilon(g) = 1$ for all $g\in \Delta^+$, we have $\Delta^+ \leq A_\delta^\dagger$ for any ideal $A$. We say that $A$ is \emph{$\delta$-faithful} if $A_\delta^\dagger = \Delta^+$ (and $A$ is \emph{$\delta$-unfaithful} if $A_\delta^\dagger$ is infinite).
\end{defn}

\begin{lem}\label{lem: delta-faithful}
The following are equivalent for any ideal $A$ of $kG$ containing $\mathfrak{M}$:
\begin{itemize}
\item[(i)] $A$ is almost faithful (as an ideal of $kG$).
\item[(ii)] $A$ is $\delta$-faithful.
\item[(iii)] $\mathfrak{a}$ is faithful (as an ideal of $k'[[G/\Delta^+]]$).
\end{itemize}
\end{lem}

\begin{proof}
$ $

$\boxed{\text{(ii)} \Leftrightarrow \text{(iii)}}$
By the above calculation, we see that $A$ is defined to be a $\delta$-faithful ideal of $kG$ precisely when $\mathfrak{a}$ is a faithful ideal of $k'[[G/\Delta^+]]$.

$\boxed{\text{(i)} \Leftrightarrow \text{(ii)}}$
Let $m = |\mathrm{im}(\delta)|$. Note that $m < \infty$ as $k$ is assumed to be a finite field. Then $\delta(g^m) = e$ for all $g\in G$, so

\begin{align*}
g^m \in A_\delta^\dagger & \Leftrightarrow \delta(g^m)\overline{g^m}^{-1} - e \in e\cdot\overline{A} \\
& \Leftrightarrow e\overline{g^m}^{-1} - e \in e\cdot\overline{A} \\
& \Leftrightarrow e (\overline{g^m}^{-1} - 1) \in e\cdot\overline{A} \\
& \Leftrightarrow \overline{g^m}^{-1} - \overline{1} \in \overline{A} \\
& \Leftrightarrow g^{-m} - 1 \in A + JkG = A,
\end{align*}

so writing $(A^\dagger)^m := \langle g^m | g\in A^\dagger\rangle$, and likewise $(A_\delta^\dagger)^m$, we see that these two subgroups are equal.

Now, suppose $A$ is almost faithful, and so in particular $A^\dagger$ is torsion; then the subgroup $(A^\dagger)^m = (A_\delta^\dagger)^m$ is also torsion, and since $g^m$ is torsion for any $g\in A_\delta^\dagger$, we have that $g$ must also be torsion. So $A_\delta^\dagger$ is a torsion subgroup of $G$. Hence it must be finite: indeed, given any open normal uniform subgroup $U$ of $G$, the kernel of the composite map $A_\delta^\dagger \hookrightarrow G \to G/U$ is a subgroup of $U\cap A_\delta^\dagger$, which is trivial as $U$ is torsion-free \cite[4.5]{DDMS}. So $A_\delta^\dagger$ embeds into the finite group $G/U$, and as $A_\delta^\dagger$ is also normal in $G$ by Definition \ref{defn: delta-dagger}, it is a finite orbital subgroup of $G$ and hence must be a subgroup of $\Delta^+$, i.e. $A$ is $\delta$-faithful. The converse is similar.
\end{proof}

In summary:

\begin{propn}\label{propn: untwisted faithful primes remain faithful primes}
Adopt the notation of Definition \ref{defn: corresps}, and recall the correspondences (I--III) there. Suppose that $e,\mathfrak{M}$ and $\delta$ are as in Definition \ref{defn: corresps}(II), so that all three correspondences (I--III) hold. Let $A$ be an ideal of $kG$ containing $\mathfrak{M}$, and let $\psi\circ q(A) = M_t(\mathfrak{a})$, so that
\begin{equation*}
\begin{minipage}{0.9\textwidth}
\xymatrix{
A \ar@{<->}[r]^-{\scriptsize \txt{(I)}}&
e\cdot\overline{A}\ar@{<->}[r]^-{\scriptsize \txt{(II)}}&
M_t(\mathfrak{a}) \ar@{<->}[r]^-{\scriptsize \txt{(III)}}&
\mathfrak{a}
}
\end{minipage}
\end{equation*}
is a corresponding sequence of ideals. Then $A$ is almost faithful if and only if $\mathfrak{a}$ is faithful.
\end{propn}

\begin{proof}
This follows from Lemma \ref{lem: delta-faithful}(i)$\Leftrightarrow$(iii).
\end{proof}

\section{Untwisting quotients by minimal primes}

\subsection{An untwisting map when $G$ is finite-by-(pro-$p$) and $e$ is $G$-invariant}\label{section: finite-by-pro-p}

\textbf{Throughout this subsection, unless stated otherwise, we adopt Setup \ref{setup: k, G, bar}.} We also take $\mathfrak{M}$ to be a minimal prime of $kG$, and fix $e\in\cpi(\mathfrak{M})$, which we assume is centralised by $G$. It follows that $e\cdot\overline{k\Delta^+} \cong M_t(k')$ by Lemma \ref{lem: definition of e}(iii). In particular, $e\cdot\overline{k\Delta^+}$ is a simple $k'$-algebra, and so its automorphisms are all inner by the Skolem-Noether theorem \cite[\href{http://stacks.math.columbia.edu/tag/074P}{Tag 074P}]{stacks-project}.

In this subsection, we introduce a group $A_H$ for each closed subgroup $H$ of $G$. By studying the structure of the group $A_G$ in the special case when $G/\Delta^+$ is pro-$p$, we will find an untwisting map $\delta: G\to (e\cdot\overline{k\Delta^+})^\times$ satisfying the conditions of Corollary \ref{cor: untwisting map delta} in this case.

Note that both $G$ and $(e\cdot\overline{k\Delta^+})^\times$ act on the ring $e\cdot\overline{k\Delta^+}$ by conjugation, and so we get group homomorphisms $G\to \mathrm{Inn}(e\cdot\overline{k\Delta^+})$ and $(e\cdot\overline{k\Delta^+})^\times \to \mathrm{Inn}(e\cdot\overline{k\Delta^+})$.

\begin{defn}\label{defn: fibre products A_H}
For any closed subgroup $H \leq G$, define $A_H$ to be the fibre product of $H$ and $(e\cdot\overline{k\Delta^+})^\times$ over $\mathrm{Inn}(e\cdot\overline{k\Delta^+})$ with respect to the above maps,
$$A_H = (e\cdot\overline{k\Delta^+})^\times \underset{\mathrm{Inn}(e\cdot\overline{k\Delta^+})}{\times} H,$$
a subgroup of $(e\cdot\overline{k\Delta^+})^\times \times H$. Write the projection map onto the second factor as $\pi_H: A_H \to H$.
\end{defn}

As $e\cdot\overline{k\Delta^+} \cong M_t(k')$, we have $(e\cdot\overline{k\Delta^+})^\times \cong GL_t(k')$. The centre of $(e\cdot\overline{k\Delta^+})^\times$ is therefore isomorphic to $Z(GL_t(k'))$, which we will identify with $k'^\times$ under the diagonal embedding. In particular, $A_H$ is an extension of $H$ by $k'^\times$. Indeed, it is easy to check that the following diagram commutes and has exact rows:

\centerline{
\xymatrix{
1 \ar[r] & k'^\times \ar[r]^i\ar@{=}[d] & A_H \ar@{.>}[r]^{\pi_H}\ar@{.>}[d] & H \ar[r]\ar[d] & 1\\
1 \ar[r] & k'^\times \ar[r] & (e\cdot\overline{k\Delta^+})^\times \ar[r] & \mathrm{Inn}(e\cdot\overline{k\Delta^+}) \ar[r]& 1,
}
}

where the inclusion map $i: k'^\times \to A_H$ is given by $i(x) = (x,1)$. (The image of $i$ is just $A_{\{1\}}$.)

We will now examine the subgroup structure of $A_G$.

\begin{lem}\label{lem: quotients of fibre products A}
If $N$ is a closed normal subgroup of $G$, then $A_N \lhd A_G$, and $A_G/A_N \cong G/N$.
\end{lem}

\begin{proof}
Firstly, clearly $A_N$ is naturally a subgroup of $A_G$, and the following diagram commutes and has exact rows:

\centerline{
\xymatrix{
1 \ar[r] & k'^\times \ar[r]\ar@{=}[d] & A_N \ar[r]^{\pi_N}\ar@{^(->}[d] & N \ar[r]\ar@{^(->}[d] & 1\\
1 \ar[r] & k'^\times \ar[r] & A_G \ar[r]_{\pi_G} & G \ar[r]& 1.
}
}

Let $(r,n) \in A_N$ and $(s,g) \in A_G$. For any $x\in e\cdot\overline{k\Delta^+}$, we have $x^r = x^n$ and $x^s = x^g$ (by definition of $A_N$ and $A_G$ respectively), which implies that $x^{s^{-1}rs} = x^{g^{-1}ng}$. As $g^{-1}ng \in N$, we have $$(s,g)^{-1}(r,n)(s,g) = (s^{-1}rs, g^{-1}ng) \in A_N,$$ and so $A_N \lhd A_G$.

Hence we may take cokernels of the vertical maps, completing the above diagram to the following commutative diagram, whose columns and first two rows are exact:

\centerline{
\xymatrix{
& 1 \ar[d] & 1 \ar[d] & 1 \ar[d] &\\
1 \ar[r] & k'^\times \ar[r]\ar@{=}[d] & A_N \ar[r]\ar[d] & N \ar[r]\ar[d] & 1\\
1 \ar[r] & k'^\times \ar[r]\ar[d] & A_G \ar[r]\ar[d] & G \ar[r]\ar[d]& 1\\
1 \ar[r] & 1 \ar[r]\ar[d] & A_G/A_N \ar[r]\ar[d] & G/N \ar[r]\ar[d] & 1\\
& 1 & 1 & 1 &
}
}

By the Nine Lemma \cite[Chapter XII, Lemma 3.4]{maclane}, the third row is now also exact.
\end{proof}

Consider the natural map $\Delta^+ \to (e\cdot \overline{k\Delta^+})^\times$ given by $g\mapsto e\cdot\overline{g}$. There is a ``diagonal" inclusion map $d: \Delta^+ \to A_G$ given by $g \mapsto (e\cdot\overline{g}, g)$, and the image $d(\Delta^+)$ is normal in $A_G$: indeed, suppose we are given $(x,h)\in A_G$. Then
\begin{align*}
d(g)^{(x,h)} & = (x,h)^{-1}(e\cdot\overline{g},g)(x,h) & \\
& = ((e\cdot\overline{g})^x, g^h) & \\
& = ((e\cdot\overline{g})^h, g^h) & \text{by definition of } A_G \\
& = d(g^h). &
\end{align*}

\begin{rk}\label{rk: A_Delta+}
The map $d$ has image contained in $A_{\Delta^+}$. Considered as a map $d: \Delta^+\to A_{\Delta^+}$, $d$ splits the map $\pi_{\Delta^+}: A_{\Delta^+}\to \Delta^+$.

Hence there are copies of $k'^\times$ and $\Delta^+$ in $A_{\Delta^+}$, and they commute: given $x\in k'^\times$, $g\in\Delta^+$, we have $$i(x) d(g) = (x,1)(e\cdot\overline{g},g) = (e\cdot\overline{g},g)(x,1) = d(g) i(x)$$ as $x$ commutes with $e\cdot\overline{g}$ inside $e\cdot\overline{k\Delta^+}$. We also have that $i(k'^\times)\cap d(\Delta^+) = \{(e, 1)\}$ is the trivial subgroup of $A_{\Delta^+}$.

In other words, $A_{\Delta^+} = i(k'^\times) d(\Delta^+) \cong k'^\times \times \Delta^+.$
\end{rk}

\begin{lem}\label{lem: defn of sigma}
Suppose that we have an injective group homomorphism $\sigma: G\to A_G$ splitting $\pi_G$ such that, for all $g\in \Delta^+$, we have $\sigma(g) = (e\cdot\overline{g}, g)$. Then we can find a group homomorphism $\delta:G\to (e\cdot\overline{k\Delta^+})^\times$ satisfying the conditions of Corollary \ref{cor: untwisting map delta}.
\end{lem}

\begin{proof}
Define $\delta$ to be the composite of $\sigma: G\to A_G$ with the projection $A_G \to (e\cdot\overline{k\Delta^+})^\times$.
\end{proof}

We suppose now that $G/\Delta^+$ is pro-$p$, and find a map $\sigma$ satisfying Lemma \ref{lem: defn of sigma} in this case.

Write $P = G/\Delta^+$, a pro-$p$ group, and note that this is isomorphic to $A_G/A_{\Delta^+}$ by Lemma \ref{lem: quotients of fibre products A}. Note also that $A_{\Delta^+}/d(\Delta^+) \cong k'^\times$ by Remark \ref{rk: A_Delta+}, and as we have assumed that $k$ is finite, we know that this is a $p'$-group. Hence $A_G/d(\Delta^+)$ is, up to isomorphism, an extension of the pro-$p$ group $P$ by the finite $p'$-group $k'^\times$; and so we may apply Sylow's theorems \cite[\S 1, exercise 11]{DDMS} to find a Sylow pro-$p$ subgroup $L/d(\Delta^+)$ of $A_G/d(\Delta^+)$ which is isomorphic to $P$.

This information is summarised in the following diagram.

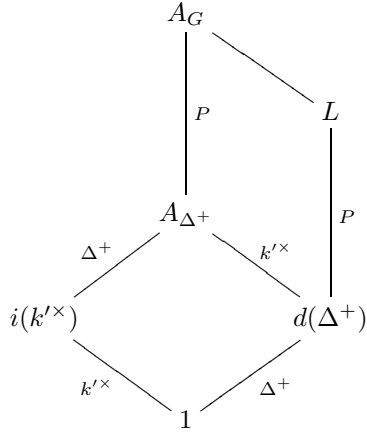
\centerline{
\xymatrix{
&A_G\ar@{-}[dd]^P\ar@{-}[dr]&\\
&&L\ar@{-}[dd]^P\\
&A_{\Delta^+}\ar@{-}[dr]^{k'^\times}\ar@{-}[dl]_{\Delta^+}&\\
i(k'^\times)\ar@{-}[dr]_{k'^\times}&&d(\Delta^+)\ar@{-}^{\Delta^+}[dl]\\
&1&
}
}

\begin{lem}\label{lem: sigma exists}
Suppose $G/\Delta^+$ is pro-$p$. Then there is a group homomorphism $\sigma$ splitting the surjection $\pi_G: A_G \to G$ satisfying the hypotheses of Lemma \ref{lem: defn of sigma}.
\end{lem}

\begin{proof}
Consider $\pi_G|_L: L\to G$. Now, $\ker(\pi_G) = i(k'^\times)$, so
\begin{equation*}
\begin{split}
\ker(\pi_G|_L) & = i(k'^\times) \cap L \\
& =(i(k'^\times) \cap A_{\Delta^+}) \cap L \\
& = i(k'^\times) \cap (A_{\Delta^+} \cap L) \\
& = i(k'^\times) \cap d(\Delta^+) = 1,
\end{split}
\end{equation*}
so $\pi_G|_L$ is injective. Also,
\begin{equation*}
i(k'^\times) \cdot L = A_{\Delta^+} \cdot L = A_G,
\end{equation*}
so
\begin{align*}
G = \pi_G(A_G) & = \pi_G(i(k'^\times) \cdot L) & \\
& = \pi_G(i(k'^\times)) \cdot \pi_G(L) & \\
& = \pi_G(L) &
\end{align*}
as $\pi_G(i(x)) = \pi_G((x,1)) = 1$ for $x\in k'^\times$, and hence $\pi_G|_L$ is surjective. So $\pi_G|_L$ is in fact an isomorphism $L\to G$.

Define $\sigma: G\to L\to A_G$ (i.e. $(\pi_G|_L)^{-1}$ followed by inclusion). By construction, this $\sigma$ is a homomorphism, and splits $\pi_G$. Also, as $\pi_G(\sigma(g)) = \pi_G(d(g)) = g$ for all $g\in \Delta^+$, we have that $\sigma(g)d(g)^{-1} \in \ker \pi_G \cap L = 1$, and so $\sigma(g) = d(g) = (e\cdot\overline{g},g)$ for $g\in \Delta^+$ as required.
\end{proof}

Now we may define $\delta: G \to (e\cdot\overline{k\Delta^+})^\times$ as in the proof of Lemma \ref{lem: defn of sigma}, allowing us to deduce the following theorem, in which we continue to write $q: kG \to e\cdot\overline{kG}$ for the natural quotient map:

\begin{thm}\label{thm: we can find a delta when G/Delta^+ is pro-p}
Adopt Setup \ref{setup: k, G, bar}, and assume that $G/\Delta^+$ is pro-$p$. Let $\mathfrak{M}$ be a minimal prime of $kG$, and $e \in \cpi(\mathfrak{M})$, and assume further that $e$ is centralised by $G$.
\begin{itemize}
\item[(i)] There exists an isomorphism
\begin{equation*}
\Psi: e\cdot\overline{kG} \to e\cdot\overline{k\Delta^+}\tensor{k} k[[G/\Delta^+]].
\end{equation*}
\item[(ii)] There exist a finite field extension $k'/k$ and a positive integer $t$, and an isomorphism
\begin{equation*}
\psi: e\cdot\overline{kG} \to M_t(k'[[G/\Delta^+]]).
\end{equation*}
\end{itemize}
Furthermore, let $A$ be an ideal of $kG$ with $\mathfrak{M}\subseteq A$, so that $\psi\circ q(A) = M_t(\mathfrak{a})$ for some ideal $\mathfrak{a}$ of $k'[[G/\Delta^+]]$. Then:
\begin{itemize}
\item[(iii)] $A$ is prime if and only if $\mathfrak{a}$ is prime. Also, $A$ is almost faithful (as an ideal of $kG$) if and only if $\mathfrak{a}$ is faithful (as an ideal of $k[[G/\Delta^+]]$).
\end{itemize}
\end{thm}

\begin{proof}
$ $

\begin{itemize}
\item[(i)] The map $\delta:G\to (e\cdot\overline{k\Delta^+})^\times$ given by Lemmas \ref{lem: defn of sigma} and \ref{lem: sigma exists} satisfies the conditions of Corollary \ref{cor: untwisting map delta}, which gives the isomorphism $\Psi: e\cdot\overline{kG} \to e\cdot\overline{k\Delta^+}\tensor{k} k[[G/\Delta^+]]$.
\item[(ii)] As in Corollary \ref{cor: untwisting map delta}, we may identify $e\cdot\overline{k\Delta^+}\tensor{k} k[[G/\Delta^+]]$ with $M_t(k'[[G/\Delta^+]])$ by appealing to Lemma \ref{lem: definition of e}(iii). Now this result follows from (i).
\item[(iii)] This is just Proposition \ref{propn: untwisted faithful primes remain faithful primes}.\qedhere
\end{itemize}
\end{proof}

\textit{Proof of Theorem A.} Writing $\overline{(-)}$ for the natural quotient map $kG\to kG/J(k\Delta^+)kG$ puts us in Setup \ref{setup: k, G, bar}. As $\mathfrak{M}\cap k\Delta^+$ is prime, we see from Lemma \ref{lem: interpretation of G centralising e} that $G$ must centralise $e$. Now the result follows from Theorem \ref{thm: we can find a delta when G/Delta^+ is pro-p}(ii).\qed

\subsection{Central twists by 2-cocycles}

First, some basic definitions.

\begin{defn}
Let $A$ be a (not necessarily commutative) ring, and $G$ an arbitrary group. Suppose we are given two functions \emph{of sets}, $\sigma: G\to \Aut(A)$ and $\alpha: G\times G\to A^\times$. (We will write the image of $a\in A$ under the automorphism $\sigma(g)$ as $a^{\sigma(g)}$.) If, for all $x, y, z\in G$, we have
\begin{equation}\label{eqn: 2-cocycle}
\alpha(xy,z)\alpha(x,y)^{\sigma(z)} = \alpha(x,yz)\alpha(y,z),
\end{equation}
then we will say that $\alpha$ is a \emph{2-cocycle (with respect to $\sigma$)}. We will write the set of such functions $\alpha$ as $Z_\sigma^2(G, A)$.
\end{defn}

Let $G$ be a finite group, $R$ a ring, and $S = R*G$ a fixed crossed product. Recall \cite[\S 1]{passmanICP} that this means that:
\begin{itemize}
\item $S$ is a free $R$-module on a generating set $\overline{G} \subseteq S$, where $\overline{G}$ comes equipped with a fixed bijection $G\to\overline{G}$, which we call a \emph{basis} for $S$ (and may simply write as $\overline{G}$ if the bijection is not important); and
\item writing our chosen basis as $g\mapsto \overline{g}$ for all $g\in G$, multiplication in $S$ is given by
\begin{align}
r\overline{g} &= \overline{g} r^{\sigma(g)}&& \text{for all } r\in R, g\in G, \label{eqn: passman conj}\\
\overline{g}\overline{h} &= \overline{gh}\tau(g,h)&& \text{for all } g,h\in G,\label{eqn: passman mult}
\end{align}
where
\begin{align}
\sigma: G\to \Aut(R),&&\text{the \emph{action}}, \nonumber\\
\tau: G\times G\to R^\times,&&\text{the \emph{twisting}},\nonumber
\end{align}
are two functions \emph{of sets} satisfying
\begin{align}
\sigma(x)\sigma(y) &= \sigma(xy)\eta(x,y)\label{eqn: passman sigma}\\
\tau(xy,z)\tau(x,y)^{\sigma(z)} &= \tau(x,yz)\tau(y,z),\label{eqn: passman tau}
\end{align}
where $\eta(x,y)$ is the automorphism of $R$ given by conjugation by $\tau(x,y)$.
\end{itemize}

\begin{rks}\label{rk: cp has ring structure}
$ $

\begin{enumerate}
\item Equation (\ref{eqn: passman tau}) says that $\tau$ is a 2-cocycle for $\sigma$ with values in $R^\times$.
\item Equations (\ref{eqn: passman sigma}) and (\ref{eqn: passman tau}) are precisely the conditions on $\sigma$ and $\tau$ for which the multiplication defined in equations (\ref{eqn: passman conj}) and (\ref{eqn: passman mult}) gives $S$ a ring structure. See \cite[\S 1]{passmanICP} for more details.
\end{enumerate}
\end{rks}

\begin{notn}\label{notn: crossed products}
We will often need to write this structure explicitly as $$S = R\cp{\sigma}{\tau} G.$$
\end{notn}
Now, we state the setup of the rest this section. Suppose we are given a ring $R$, a finite group $G$, and a fixed crossed product
$S = R\cp{\sigma}{\tau} G$
as in Notation \ref{notn: crossed products}; and suppose further that we wish to define some new crossed product of $R$ by $G$, with the same action map $\sigma$ but a different twisting map $\tau'$, i.e. $S' = ``R\cp{\sigma}{\tau'} G"$. (We will write the bases of these two crossed products differently for clarity.)

Explicitly: write $S$ as the free left $R$-module $\bigoplus_{g\in G} R\overline{g}$, with basis $g\mapsto \overline{g}$, together with multiplication data $\sigma$ and $\tau$ satisfying equations (\ref{eqn: passman conj}--\ref{eqn: passman tau}). Let $S'$ be another free left $R$-module, $\bigoplus_{g\in G} R\hat{g}$, with basis $g\mapsto \hat{g}$, together with the data of a map $\sigma$ satisfying (\ref{eqn: passman conj}) and a map $\tau'$ satisfying (\ref{eqn: passman mult}) with respect to $\sigma$.

Of course, \emph{a priori}, $S'$ may not satisfy (\ref{eqn: passman sigma}--\ref{eqn: passman tau}), so the above may not define a ring structure on $S'$. So we ask: under which conditions on $\tau'$ is $S'$ a ring?

\textbf{For the rest of this section, until stated otherwise, we fix a ring $R$, a finite group $G$ and a crossed product $S = R\cp{\sigma}{\tau} G$, and we write $A = Z(R)^\times$.}

\begin{lem}
$S'$ is well-defined as a ring if and only if there exists $\alpha\in Z^2_{\sigma}(G,A)$ satisfying $\tau'(x,y) = \tau(x,y) \alpha(x,y)$ for all $x, y\in G$.
\end{lem}

\begin{proof}
By Remark \ref{rk: cp has ring structure}.2, we need to check that $\alpha\in Z^2_{\sigma}(G,A)$ if and only if $S'$ satisfies (\ref{eqn: passman sigma}--\ref{eqn: passman tau}).

In order for $S$ and $S'$ to satisfy (\ref{eqn: passman sigma}), we must have
\begin{align*}
\sigma(x)\sigma(y) &= \sigma(xy)\eta(x,y)&\text{ and}\\
\sigma(x)\sigma(y) &= \sigma(xy)\eta'(x,y)& \text{ for all }x, y\in G,
\end{align*}
where $\eta(x,y)$ and $\eta'(x,y)$ are the automorphisms induced by conjugation by $\tau(x,y)$ and $\tau'(x,y)$ respectively. This is equivalent to the statement that $\eta = \eta'$. In other words, writing $\alpha = \tau^{-1}\tau'$ pointwise, we see that conjugation by $\alpha(x,y)$ must induce the trivial automorphism on $R$, and so $$\alpha:G\times G\to Z(R)^\times = A.$$
In order for $\tau'$ to satisfy (\ref{eqn: passman tau}), we must have that $\alpha$ satisfies (\ref{eqn: passman tau}), i.e. $\alpha\in Z^2_{\sigma}(G,A)$. The converse is obvious.
\end{proof}

\begin{defn}\label{defn: cocycle twist}
When the crossed product $S = R*G = R\cp{\sigma}{\tau} G$ and the central 2-cocycle $\alpha$ are fixed, write the ring $S'$ defined above as $S_\alpha$: we will say that $S_\alpha$ is the \emph{central 2-cocycle twist of $S$ by $\alpha$ with respect to the decomposition $S = R\cp{\sigma}{\tau} G$}, meaning that
$$S_\alpha = R\cp{\sigma}{\tau\alpha} G.$$
Sometimes it will not be necessary to specify all of this information in full; we may simply refer to $S_\alpha$ as \emph{a central 2-cocycle twist of $S$}, or similar.
\end{defn}

\begin{rks}
$ $

\begin{enumerate}
\item Note that $S_\alpha$ depends not only on the map $\tau$, but also on the choice of basis $G\to \overline{G}$ for $S = R*G$.
\item Fix a crossed product $S = R*G$, choose some $\alpha\in Z^2_\sigma(G,A)$, and form the 2-cocycle twist $S_\alpha$. Write the resulting crossed product decompositions as
$$S = \bigoplus_{g\in G} R\overline{g},\;\; S_\alpha = \bigoplus_{g\in G} R\hat{g},$$
with bases $g\mapsto \overline{g}$ and $g\mapsto\hat{g}$ respectively.

We say that $S$ and $S_\alpha$ differ by a \emph{diagonal change of basis} if there is a set of units $\{u_g\in R^\times: g\in G\}$ such that $\hat{g} \mapsto \overline{g} u_g$ gives an $R$-linear ring isomorphism $S_\alpha\to S$. (In particular, if $S$ and $S_\alpha$ differ by a diagonal change of basis, they are isomorphic as rings.) By \cite[exercise 1.1]{passmanICP}, $S$ and $S_\alpha$ differ by a diagonal change of basis if and only if $\alpha$ is a \emph{2-coboundary} for $\sigma$, i.e. there is some function $\varphi: G\to R^\times$ with $$\alpha(x,y) = \varphi(xy)^{-1} \varphi(x)^{\sigma(y)} \varphi(y)$$ for all $x,y\in G$. Hence $S$ and $S_\alpha$ are non-isomorphic only if $\alpha$ has non-trivial cohomology class. But we will not develop this idea any further in this paper.
\item We note that similar twists have been studied by Aljadeff \emph{et al.}, e.g. in \cite{aljadeff-rob}.
\end{enumerate}
\end{rks}

Central 2-cocycle twists will occur naturally in the theory later. First, we record below an easy consequence of the definitions above, namely that the twisting operation $(-)_\alpha$ preserves primality in certain special cases. We are certain that this will be useful in future work on prime extension rings.

\begin{defn}\label{defn: X-inn autom}{\cite[Lemma 12.3]{passmanICP}}
Let $R$ be a prime ring. An automorphism $\varphi: R\to R$ is \emph{X-inner} if there exist nonzero elements $a,b,c,d\in R$ such that, for all $x\in R$,
$$axb = cx^\varphi d.$$
(Here $x^\varphi$ denotes the image of $x$ under $\varphi$.) Write $\mathrm{Xinn}(R)$ to denote the subgroup of $\Aut(R)$ of X-inner automorphisms.

Now let $G$ be a group, and fix a crossed product $S = R\cp{\sigma}{\tau} G$. Write $\Xinn{S}{R}{G}$ for the normal subgroup of $G$ consisting of elements $g\in G$ that act by X-inner automorphisms on $R$, i.e.
$$\Xinn{S}{R}{G} = \sigma^{-1}(\sigma(G)\cap \mathrm{Xinn}(R)).$$
\end{defn}

\begin{thm}
Fix a crossed product $S = R*G$ with $R$ a prime ring and $G$ a finite group. Then $\Xinn{S}{R}{G}$ and $\Xinn{S_\alpha}{R}{G}$ are equal for every $\alpha\in Z^2_\sigma(G,A)$. In particular, if $\Xinn{S}{R}{G} = 1$, then $S_\alpha$ is also a prime ring for every $\alpha\in Z^2_\sigma(G,A)$.
\end{thm}

\begin{proof}
It is clear from the definition that $\Xinn{S_\alpha}{R}{G}$ depends only on the map $\sigma$, and so $\Xinn{S_\alpha}{R}{G} = \Xinn{S}{R}{G}$ for all $\alpha$. A special case of \cite[Corollary 12.6]{passmanICP} implies that, if $\Xinn{S_\alpha}{R}{G} = 1$, then $S_\alpha$ is a prime ring.
\end{proof}

The first half of this theorem is little more than a trivial manipulation of terminology, but we state it in full as we will make crucial use of this fact in a future paper.

\subsection{A crossed product structure for general $G$}

Now we turn our attention back to the problem of understanding quotients of completed group algebras.

Adopt Setup \ref{setup: k, G, bar}. Let $\mathfrak{M}$ be a minimal prime of $kG$, $e\in\cpi(\mathfrak{M})$, and suppose that $e$ is centralised by $G$.

We are now working in a more general case than \S 4.1, as $G/\Delta^+$ is not necessarily pro-$p$. In this case we may not be able to find a group homomorphism $\delta: G \to (e\cdot\overline{k\Delta^+})^\times$ satisfying the hypotheses of Corollary \ref{cor: untwisting map delta}, and so the conclusions of Theorem \ref{thm: we can find a delta when G/Delta^+ is pro-p} may not hold. In particular, $e\cdot\overline{kG}$ may not be isomorphic to $M_t(k'[[G/\Delta^+]])$, as might be expected.

In this case, fix an open normal pro-$p$ subgroup $N$ of $G/\Delta^+$ (e.g. by taking the normal core in $G$ of an open uniform subgroup, as in \cite[Theorem 8.32]{DDMS}), and write $H$ for the preimage of $N$ in $G$, so that by Theorem \ref{thm: we can find a delta when G/Delta^+ is pro-p} we \emph{do} get isomorphisms
\begin{eqnarray*}
&\Psi: e\cdot\overline{kH} \to e\cdot\overline{k\Delta^+}\tensor{k'} k'N,\\
&\psi: e\cdot\overline{kH} \to M_t(k'N).
\end{eqnarray*}
(Here, in the codomain of $\Psi$, we have identified the two rings $e\cdot\overline{k\Delta^+}\tensor{k} kN$ and $e\cdot\overline{k\Delta^+}\tensor{k'} k'N$ as in Lemma \ref{lem: definition of e}(iii).)

Now, to understand the structure of the ring $e\cdot\overline{kG}$, we will have to rely on a suitable crossed product structure of $kG$. That is, writing $F = G/H$, we can find a crossed product decomposition
\begin{equation*}
kG = kH * F.
\end{equation*}
In the following discussion, we will construct a related crossed product $k'N * F$ (\emph{not} necessarily isomorphic to $k'[[G/\Delta^+]]$!), and show that the isomorphism $\psi$ extends to an isomorphism
\begin{equation*}
\tilde{\psi}: e\cdot\overline{kG} \to M_t(k'N*F).
\end{equation*}
In fact, the crossed product $k'N*F$ that we construct will naturally be a central 2-cocycle twist of $k'[[G/\Delta^+]]$.

Recall the map $\delta: H\to (e\cdot\overline{k\Delta^+})^\times$ constructed by applying Lemmas \ref{lem: defn of sigma} and \ref{lem: sigma exists} to $H$. Continue to write
\begin{align*}
\varepsilon: H&\to (e\cdot\overline{kH})^\times\\
h&\mapsto \delta(h)^{-1}\overline{h}
\end{align*}
for all $h\in H$, as in the proof of Theorem \ref{thm: untwisting map delta}. Recall that this $\varepsilon$ is a continuous group homomorphism.

\textbf{For the remainder of this subsection, we fix an element $g\in G$.}

Fix $M_g \in (e\cdot\overline{k\Delta^+})^\times$, an arbitrary lift of the image of $g$ under the map $G\to \mathrm{Inn}(e\cdot\overline{k\Delta^+})$. In other words, $M_g$ is any element of $(e\cdot\overline{k\Delta^+})^\times$ such that $x^g = x^{M_g}$ for all $x\in e\cdot\overline{k\Delta^+}$, or equivalently such that $(M_g,g)\in A_G$. Define
\begin{equation}\label{eqn: g-tilde}
\widetilde{g} = M_g^{-1} \overline{g} \in (e\cdot\overline{kG})^\times
\end{equation}
-- by construction, $\widetilde{g}$ centralises $e\cdot\overline{k\Delta^+}$, and this element will play the role of ``$\varepsilon(g)$" when $g\not\in H$.

Recall also the isomorphisms

\centerline{
\xymatrix{
e\cdot\overline{kH} \ar@<1ex>[r]^-{\Psi} \ar@{<-}@<-1ex>[r]_-{\Phi}& e\cdot\overline{k\Delta^+} \tensor{k} kN
}
}

of the proof of Theorem \ref{thm: untwisting map delta} and Corollary \ref{cor: untwisting map delta}.

Conjugation on the right by $\widetilde{g}$ is a ring automorphism $\varphi_g \in \mathrm{Aut}(e\cdot\overline{kH})$, which induces a ring automorphism
$$\Psi\circ \varphi_g\circ \Phi =: \theta_g \in \mathrm{Aut}(e\cdot\overline{k\Delta^+}\tensor{k} kN).$$

That is:

\centerline{
\xymatrix{
e\cdot\overline{kH} \ar[r]^-{\varphi_g}& e\cdot\overline{kH}\ar[d]^\Psi\\
e\cdot\overline{k\Delta^+}\tensor{k} kN\ar[u]^\Phi \ar[r]_-{\theta_g} & e\cdot\overline{k\Delta^+}\tensor{k} kN
}
}

For ease of notation, we will write these maps on the left as before, though note that e.g. $\theta_{g_1}\circ \theta_{g_2}(x) = \theta_{g_2 g_1}(x)$ for all $g_1, g_2\in G$.

Given $r\in e\cdot\overline{k\Delta^+}$ and $h\Delta^+ \in N$, we wish to calculate $\theta_g(r\otimes h\Delta^+)$ explicitly.

We begin with a trivial remark:

\begin{lem}\label{lem: f-tilde centralises x}
By construction, $\varphi_g(r) = r$, and so $\theta_g(r\otimes 1) = r\otimes 1$. \qed
\end{lem}

Next, a computational lemma:

\begin{lem}\label{lem: action of f commutes with delta}
$\delta(h)^g = \delta(h^g)$ for all $h\in H$.
\end{lem}

\begin{proof}
Define $\beta_g: H \to (e\cdot\overline{k\Delta^+})^\times$ by $\beta_g(h) = \delta(h^g)^{-1}\delta(h)^g$. We aim to show that $\beta_g(h) = 1$ for all $h$.

For any $r\in e\cdot\overline{k\Delta^+}$, we have that
\begin{align*}
\begin{split}
r^{\delta(h^g)} &= r^{h^g} \\
& = ((r^{g^{-1}})^h)^g \\
& = ((r^{g^{-1}})^{\delta(h)})^g = r^{\delta(h)^g},
\end{split}
\end{align*}
and so $r^{\beta_g(h)} = r$, i.e. $\beta_g(h)$ is in the centre of $(e\cdot\overline{k\Delta^+})^\times$. So $\beta_g$ is a map from $H$ to $k'^\times$.

Let $h_1, h_2 \in H$. Since $\beta_g(h_1) = \delta(h_1^g)^{-1}\delta(h_1)^g$ is central in $(e\cdot\overline{k\Delta^+})^\times$, in particular it centralises $\delta(h_2^g)^{-1}$, and so
\begin{align*}
\beta_g(h_1h_2) & = \delta((h_1h_2)^g)^{-1} \delta(h_1h_2)^g\\
& = \delta(h_2^g)^{-1}\delta(h_1^g)^{-1}\delta(h_1)^g\delta(h_2)^g\\
& = \boxed{\delta(h_2^g)^{-1}}\,\boxed{\delta(h_1^g)^{-1}\delta(h_1)^g}\,\delta(h_2)^g\\
& = \boxed{\delta(h_1^g)^{-1}\delta(h_1)^g}\,\boxed{\delta(h_2^g)^{-1}}\,\delta(h_2)^g \\
& = \beta_g(h_1) \beta_g(h_2),
\end{align*}
so $\beta_g$ is a group homomorphism $H\to k'^\times$. Furthermore, when $h\in \Delta^+$,
\begin{equation*}
\begin{split}
\beta_g(h) & = \delta(h^g)^{-1}\delta(h)^g \\
& = (e\cdot\overline{h^g})^{-1} (e\cdot\overline{h})^g = 1.
\end{split}
\end{equation*}
So $\Delta^+ \leq \ker \beta_g$, and so $\beta_g$ in fact descends to a homomorphism from $N$ to $k'^\times$. However, by assumption, $N$ is pro-$p$, and $k'^\times$ is a finite $p'$-group, so $\beta_g$ must be trivial.
\end{proof}

\begin{rk}
Of course, this result does not require $k$ to be finite -- the only fact we have used about $k$ is that $k'^\times$ contains no non-trivial pro-$p$ subgroups. But we do not need this generality here.
\end{rk}

Continue to write $\varepsilon(h) = \delta(h)^{-1}\overline{h}$ for all $h\in H$. Then, finally, we can conclude:

\begin{cor}\label{cor: alpha and f-tilde}
$\varepsilon(h)^{\widetilde{g}} = \varepsilon(h^g)$ for all $h\in H$.
\end{cor}

\begin{proof}
We have
\begin{align*}
\varepsilon(h)^{\widetilde{g}} & = \varepsilon(h)^{g} & \text{as } \varepsilon(h) \text{ centralises } M_g \\
& = (\delta(h)^{g})^{-1} \overline{h^{g}} & \\
& = (\delta(h^{g}))^{-1} \overline{h^{g}} & \text{ by Lemma \ref{lem: action of f commutes with delta}} \\
& = \varepsilon(h^{g}),&
\end{align*}
as required.
\end{proof}

Now we can calculate the action of $\theta_g$ on $e\cdot\overline{k\Delta^+}\tensor{k} kN$:

\begin{lem}\label{lem: action of f-tilde}
Given $r\in e\cdot\overline{k\Delta^+}$ and $h\Delta^+ \in N$, we have $$\theta_g(r\otimes h\Delta^+) = r\otimes (h\Delta^+)^{g\Delta^+}.$$
\end{lem}

\begin{proof}
\begin{align*}
\theta_g(r\otimes h\Delta^+) & = \Psi(\varphi_g(\Phi(r\otimes h\Delta^+))) & \\
& = \Psi(\varphi_g(r\varepsilon(h))) & \text{by definition} \\
& = \Psi(r^{\widetilde{g}} \varepsilon(h)^{\widetilde{g}}) & \\
& = \Psi(r \varepsilon(h^g)) & \text{by Corollary \ref{cor: alpha and f-tilde} and Lemma \ref{lem: f-tilde centralises x}} \\
& = \Psi(\Phi(r\otimes h^g \Delta^+)) & \text{by definition} \\
& = r\otimes h^g \Delta^+ & \text{by Theorem \ref{thm: untwisting map delta}} \\
& = r\otimes (h\Delta^+)^{g\Delta^+}. &\qedhere
\end{align*}
\end{proof}

We will need one final definition.

\begin{defn}\label{defn: standard cp}
Let $G$ be a compact $p$-adic analytic group, $H$ an open normal subgroup, $F = G/H$, $k$ a commutative pseudocompact ring, and $I$ a $G$-stable ideal of $kH$. Suppose we are given elements $x_i\in (kG/IkG)^\times$ such that
$$kG/IkG = \bigoplus_{i=1}^m (kH/I)x_i$$
is a decomposition of $kG/IkG$ as a $kH/I$-module, or equivalently,
$$\overline{F} := \{x_1, \dots, x_m\} \subseteq (kG/IkG)^\times$$
is a basis for a crossed product decomposition $kG/IkG = kH/I * F$. We will say that this decomposition is \emph{standard} if $\overline{F}$ is the image of a transversal of $H$ in $G$ under the natural map $G\to (kG/IkG)^\times$; that is, we have elements $g_1, \dots, g_m\in G$ such that $F = \{H g_1, \dots, H g_m\}$, and $x_i$ is the image of $g_i$ for all $1\leq i\leq m$.
\end{defn}

Now finally we can prove the main theorem of this subsection. We continue to use the notation established at the beginning of this subsection: $e$ is centralised by $G$, $H/\Delta^+$ is pro-$p$, and the isomorphisms $\psi$ and $\Psi$ are fixed.

Fix also a crossed product decomposition
\begin{equation}\label{eqn: fixed cp}
k'[[G/\Delta^+]] = k'[[H/\Delta^+]] \cp{\sigma}{\tau} (G/H)\tag{$\ddagger$}
\end{equation}
which is \emph{standard} in the sense of Definition \ref{defn: standard cp}.

\begin{thm}\label{thm: untwisted crossed product}
Notation as above. Then there exists
$$\alpha\in Z^2_{\sigma}\Big(G/H,\; Z(k'[[H/\Delta^+]]^\times)\Big)$$
such that $\Psi$ extends to an isomorphism
\begin{equation*}
\tilde{\Psi}: e\cdot\overline{kG} \to e\cdot\overline{k\Delta^+}\tensor{k'} \big(k'[[G/\Delta^+]]\big)_\alpha,
\end{equation*}
where the 2-cocycle twist $(k'[[G/\Delta^+]])_\alpha$, as defined in Definition \ref{defn: cocycle twist}, is taken with respect to the standard crossed product decomposition (\ref{eqn: fixed cp}) above.

Hence $\psi$ also extends to an isomorphism
\begin{equation*}
\tilde{\psi}: e\cdot\overline{kG} \to M_t\Big(\big(k'[[G/\Delta^+]]\big)_\alpha\Big).
\end{equation*}
\end{thm}

\begin{proof}
We know that $e\cdot\overline{k\Delta^+} \cong M_t(k')$ for some $t$ and $k'$ by Lemma \ref{lem: definition of e}(iii), and so $e\cdot\overline{k\Delta^+}$ contains a set $\{e_{ij}\}_{i,j=1}^t$ of $t^2$ matrix units. Set
\begin{equation*}
Z_H := Z_{e\cdot\overline{kH}}\Big(\{e_{ij}\}\Big) \subseteq e\cdot\overline{kH},
\end{equation*}
the centraliser of all of these matrix units, and likewise $Z_G \subseteq e\cdot\overline{kG}$ and $Z_{\Delta^+} \subseteq e\cdot\overline{k\Delta^+}$. Then the statement and proof of \cite[6.1.5]{passmanASGR} show that $$e\cdot\overline{kH} \cong e\cdot\overline{k\Delta^+} \tensor{Z_{\Delta^+}} Z_H$$ and $$e\cdot\overline{kG} \cong e\cdot\overline{k\Delta^+} \tensor{Z_{\Delta^+}} Z_G.$$

Since $e\cdot\overline{k\Delta^+} \cong M_t(k')$, it is clear that $Z_{\Delta^+} \cong k'$, the diagonal copy of $k'$ inside $M_t(k')$. Using the isomorphism $\Psi$ of Theorem \ref{thm: we can find a delta when G/Delta^+ is pro-p}, we can also understand the structure of $Z_H$:
\begin{align*}
\Psi(Z_H) = Z_{\Delta^+} \tensor{k} k[[H/\Delta^+]] \cong k'[[H/\Delta^+]],
\end{align*}
so $\Psi$ restricts to an explicit isomorphism $Z_H \to k'[[H/\Delta^+]]$. Now we would like to understand the structure of $Z_G$.

As $F := G/H$ is finite, and the crossed product (\ref{eqn: fixed cp}) is standard, we have $g_1, \dots, g_n\in G$ whose images form the basis of (\ref{eqn: fixed cp}). For each $g_i\in G$, form  $\widetilde{g}_i\in Z_G^\times$ as in equation (\ref{eqn: g-tilde}). Then $e\cdot\overline{kG}$ is a free $e\cdot\overline{kH}$-module of rank $n$: $e\cdot\overline{kG}$ can be written as the internal direct sum $$e\cdot\overline{kG} = \bigoplus_{i=1}^n (e\cdot\overline{kH})\widetilde{g_i}.$$ Intersecting both sides of this equation with $Z_G$ gives $$Z_G = \bigoplus_{i=1}^n Z_H\widetilde{g_i},$$ showing that $Z_G$ is a crossed product $Z_H * F$, and is therefore isomorphic to $k'[[H/\Delta^+]] * F$. Lemma \ref{lem: action of f-tilde} may now be restated to say that the map $\theta_{g_i}$, i.e. conjugation on the right by $\widetilde{g_i}$, is precisely the automorphism $\sigma(g_i)$, and hence this $k'[[H/\Delta^+]]*F$ is just a central 2-cocycle twist of the decomposition (\ref{eqn: fixed cp}) of $k'[[G/\Delta^+]]$. This is the map $\tilde{\Psi}$, and the map $\tilde{\psi}$ then also follows from Lemma \ref{lem: definition of e}(iii).
\end{proof}

\textit{Proof of Theorem B.} By Lemma \ref{lem: interpretation of G centralising e}, $\mathfrak{M}\cap kH$ is a prime ideal of $kH$, so we are in the situation of Theorem \ref{thm: untwisted crossed product}, from which Theorem B follows immediately.\qed

Retain the above notation, and let $A$ be an ideal of $kH$ with $\mathfrak{M}\subseteq A$. Continuing as before to write $q:kG\to e\cdot\overline{kG}$ for the natural quotient map, we see by Theorem \ref{thm: we can find a delta when G/Delta^+ is pro-p} that $\psi\circ q(A) = M_t(\mathfrak{a})$ for some ideal $\mathfrak{a}$ of $k'[[H/\Delta^+]]$. In this case, we have:

\begin{cor}\label{cor: G-stable ideals}
The following are equivalent.
\begin{itemize}
\item[(i)] $A$ is $G$-stable as an ideal of $kH$.
\item[(ii)] $\mathfrak{a}$ is stable as an ideal of $k'[[H/\Delta^+]]$ under the $(G/\Delta^+)$-action induced by conjugation inside $k'[[G/\Delta^+]]$.
\item[(iii)] $\mathfrak{a}$ is stable as an ideal of $k'[[H/\Delta^+]]$ under the $(G/\Delta^+)$-action induced by conjugation inside $(k'[[G/\Delta^+]])_\alpha$.
\end{itemize}
Moreover, when these conditions hold, we have
$$\tilde{\psi}\circ q(AkG) = M_t\Big(\mathfrak{a}\; \big(k'[[G/\Delta^+]]\big)_\alpha\Big).$$
\end{cor}

\begin{proof}
The equivalence of statements (ii) and (iii) is clear since, by definition, the conjugation action of $G/\Delta^+$ on the ring $k'[[G/\Delta^+]]$ is the same as the conjugation action of $G/\Delta^+$ on the ring $(k'[[G/\Delta^+]])_\alpha$. The equivalence of (i) and (ii) follows easily from Lemma \ref{lem: action of f-tilde}. Then
\begin{align*}
\tilde{\psi}\circ q(AkG) &= \big(\tilde{\psi}\circ q(A)\big)\cdot \big(\tilde{\psi}\circ q(kG)\big)\\
&= M_t(\mathfrak{a})\cdot M_t\Big(\big(k'[[G/\Delta^+]]\big)_\alpha\Big)\\
&= M_t\Big(\mathfrak{a}\; \big(k'[[G/\Delta^+]]\big)_\alpha\Big).\qedhere
\end{align*}
\end{proof}

\textit{Proof of Theorem C.} This follows from (i) Proposition \ref{propn: corresps preserve primality}, (ii) Corollary \ref{cor: G-stable ideals} and (iii) Theorem \ref{thm: we can find a delta when G/Delta^+ is pro-p}(iii) (or Proposition \ref{propn: untwisted faithful primes remain faithful primes}).

\section{The case in which $e$ is not $G$-invariant}\label{section: matrix units}

From section \ref{section: untwisting theorem} onwards, we have usually stipulated a stronger condition than in Lemma \ref{lem: definition of e}, namely that the conjugation action of $G$ on $\overline{kG}$ should fix the central idempotent $e\in\overline{k\Delta^+}$. However, in general, $e$ will have some non-trivial (but finite) $G$-orbit, so it will only make sense to consider $f\cdot\overline{kG}$, where $f = \cpisum{e}{G}$. In this section, we show that we have not lost much generality by making this assumption.

\subsection{Orbits of idempotents}

We begin with a well-known matrix units lemma.

\textbf{Throughout this subsection, we fix the following notation.} Let $R$ be a ring, and let $1 = e_1 + e_2 + \dots + e_n$ be a decomposition of $1$ into a sum of orthogonal idempotents. Let $G$ be a subgroup of the group of units of $R$, and assume that $G$ permutes the set $\{e_1, e_2, \dots, e_n\}$ transitively by conjugation: explicitly, fix elements $1 = g_1, g_2, \dots, g_n\in G$ such that $g_i e_1 = e_i g_i$, and set $g_{ij} = g_ig_j^{-1}$. Finally, for each $1\leq i,j\leq n$, write $R_{ij} = e_i R e_j$, and write $S = R_{11}$.

Recall, from Lemma \ref{lem: peirce and e-bimods}(ii), that in fact
$$R = \bigoplus_{i,j=1}^n R_{ij}$$
as an additive group, and recall the suggestive matrix notation of Remark \ref{rk: peirce decomp is a matrix ring}: we now interpret this as an honest matrix ring.

\begin{lem}\label{lem: MUL}
Define the function $\mu: R\to M_n(S)$ as follows: for each $1\leq i,j\leq n$, and for any $x\in R_{ij}$, set $\mu(x)$ equal to the matrix whose $(i,j)$-entry is $g_i x g_j^{-1}$, and whose other entries are all $0$; then extend linearly to all of $R$.
\begin{enumerate}[label=(\roman*)]
\item $\mu$ is an isomorphism of rings.
\item Let $I$ be an ideal of $R$. Then $\mu(I) = M_n(J)$, where $J = e_1 I e_1$.
\item Let $R'$ be a subring of $R$ containing the elements $e_1, \dots, e_n$, and suppose that $G\leq R'^\times$. Then $\mu(R') = M_n(S')$, where $S' = e_1 R' e_1$.
\end{enumerate}
\end{lem}

\begin{proof}
$ $

\begin{enumerate}[label=(\roman*)]
\item This is proved as \cite[6.1.6]{passmanASGR}.
\item Write $\pi$ for the $R$-bimodule projection $M_n(S) \to S$ sending $x$ to its $(1,1)$-entry, and observe that the restriction of $\pi\circ\mu$ to $R_{11}$ is simply the identity. So, by Lemma \ref{lem: peirce and e-bimods}(i), (ii), we can calculate that $\pi\circ\mu(I) = J$.

Now, $M_n(S)$ is a matrix ring, and so all of its ideals are of the form $M_n(J')$ for some ideal $J'$ of $S$. In particular, let $J'$ be such that $\mu(I) = M_n(J')$. Then $\pi(\mu(I)) = J'$, and hence $J' = J$.
\item This follows trivially by applying (i) with $R'$ in place of $R$.\qedhere
\end{enumerate}
\end{proof}

\textit{Proof of Theorem D.} By Lemma \ref{lem: definition of e}, $\mathfrak{p}$ is the full preimage in $k\Delta^+$ of $(1-e)\overline{k\Delta^+}$ for some $e\in \cpi(\mathfrak{M})$. Set $f = \cpisum{e}{G}$. Then, by Lemma \ref{lem: definition of e}, we have $kG/\mathfrak{M} \cong f\cdot\overline{kG}$, and by Lemma \ref{lem: MUL}, this is isomorphic (under $\mu$) to $M_r(e\cdot\overline{kG}\cdot e)$. It is easy to see that $e\cdot\overline{kG}\cdot e \cong e\cdot\overline{kG_1}$, where $G_1$ is the open subgroup of $G$ stabilising $e$ and hence $\mathfrak{p}$. By another application of Lemma \ref{lem: definition of e}, it is clear that $e\cdot\overline{kG_1} \cong kG_1/\mathfrak{N}$, where $\mathfrak{N}$ is the full preimage in $kG_1$ of $(1-e)\overline{kG_1}$, a minimal prime by Lemma \ref{lem: definition of e}(ii). Now composing these isomorphisms in the obvious way gives the desired isomorphism $\widetilde{\mu}: kG/\mathfrak{M} \to M_r(kG_1/\mathfrak{N})$.\qed

\subsection{A fourth ideal correspondence and primality}

\begin{defn}\label{defn: corresps 2}
Adopt Setup \ref{setup: k, G, bar}, and take $\mathfrak{M}, \mathfrak{N}, r, G_1$ as in the statement and proof of Theorem D. Now the isomorphism $\widetilde{\mu}: kG/\mathfrak{M} \to M_r(kG_1/\mathfrak{N})$ of Theorem D gives us a final ideal correspondence, in addition to those given in Definition \ref{defn: corresps}. We choose to define it in the following equivalent way:
\begin{itemize}
\item[(IV)]
We always have the ideal correspondence
\begin{equation*}
\begin{minipage}{0.9\textwidth}
\xymatrix{
\mathcal{I}_{\mathfrak{M}}(kG)\ar@{<->}[r]&
\mathcal{I}_{\mathfrak{N}}(kG_1)
}
\end{minipage}
\end{equation*}
defined to be the unique correspondence that makes the following diagram commute:

\hfil
\begin{equation}\label{eqn: (IV)}
\xymatrix{
\mathcal{I}_{\mathfrak{M}}(kG) \ar@{<..>}[d]_{\scriptsize \txt{(IV)}}\ar@{<->}[rr]^-{\scriptsize \txt{(I)}\vphantom{M}\vphantom{M}}
&&
\mathcal{I}(f\cdot\overline{kG}) \ar@{->}[d]^-{\mu}
\\
\mathcal{I}_{\mathfrak{N}}(kG_1)\ar@{<->}[r]_-{\scriptsize \vphantom{M}\txt{(I)}}
&
\mathcal{I}(e\cdot\overline{kG_1})
&
\mathcal{I}(M_n(e\cdot\overline{kG_1})),\ar@{<->}[l]^-{\scriptsize \vphantom{M}\txt{(III)}}
\\}
\end{equation}

where $\mu: f\cdot\overline{kG}\to M_n(e\cdot\overline{kG_1})$ is the isomorphism of Lemma \ref{lem: MUL}.
\end{itemize}
\end{defn}

It is now easy to deduce the following.

\begin{propn}\label{propn: IV preserves primality}
The one-to-one ideal correspondence (IV) preserves primality.
\end{propn}

\begin{proof}
Let $P$ be an ideal of $kG$ containing $\mathfrak{M}$, and $Q$ an ideal of $kG_1$ containing $\mathfrak{N}$, and suppose that $P$ and $Q$ correspond under (IV) as above. We will show that $P$ is prime if and only if $Q$ is prime by simply following them around the diagram (\ref{eqn: (IV)}).

By two applications of correspondence (I) (Definition \ref{defn: corresps}), which preserves primality by Proposition \ref{propn: corresps preserve primality}, it suffices to show that $f\cdot\overline{P}\lhd f\cdot\overline{kG}$ is prime if and only if $e\cdot\overline{Q}\lhd e\cdot\overline{kG_1}$ is prime. But, writing $\mu$ for the isomorphism $f\cdot\overline{kG}\to M_n(e\cdot\overline{kG_1})$ of Lemma \ref{lem: MUL}, we see that $f\cdot\overline{P}$ is prime if and only if $\mu(f\cdot\overline{P}) = M_n(e\cdot\overline{Q})$ is prime. Now it follows from Morita equivalence \cite[3.5.5, 3.5.9]{MR} that $M_n(e\cdot\overline{Q})$ is prime if and only if $e\cdot\overline{Q}$ is prime.
\end{proof}

\subsection{Control by subgroups}

\textbf{Notation for this subsection.} Adopt Setup \ref{setup: k, G, bar}, and let $\mathfrak{M}$ be a minimal prime ideal of $kG$. Take $e\in\cpi(\mathfrak{M})$, write $f = \cpisum{e}{G}$, and let $e = e_1, \dots, e_n$ be the $G$-orbit inside $R = f\cdot\overline{kG}$ of mutually orthogonal nonzero idempotents summing to $f = 1_R$. Let $G_1$ be the centraliser inside $G$ of $e$, and $\mathfrak{N} = (1-e)\overline{kG_1}$.

Fix also a closed normal subgroup $H$ of $G$, and set $f' = f'_1 = \cpisum{e}{H}$. Write $\mathcal{O}_G(e)$ for the $G$-orbit of $e$, and suppose that this breaks up into a disjoint union of $H$-orbits $\mathcal{O}_H(e_1) \cup \dots \cup \mathcal{O}_H(e_s)$. For each $1\leq i\leq s$, write $f'_i = \cpisum{e_i}{H}$, i.e. the sum of the elements of $\mathcal{O}_H(e_i)$.

\begin{lem}\label{lem: H-orbits are a mutually orthogonal G-orbit}
The $f'_i$ are mutually orthogonal, and form a $G$-orbit. The sum $f'_1 + \dots + f'_s = f$, and $H$ centralises each $f'_i$.
\end{lem}

\begin{proof}
The idempotents $e_j$ are all mutually orthogonal, and the $H$-orbits $\mathcal{O}_H(e_i)$ are all pairwise disjoint, so an easy calculation shows that the $f'_i$ are mutually orthogonal. It is also trivial to see that conjugation by elements of $H$ fixes $\mathcal{O}_H(e_i)$ for each $i$, and so $H$ centralises each $f'_i$.

To show that the $f'_i$ form a $G$-orbit, it is enough to show that the $\mathcal{O}_H(e_i)$ form a $G$-orbit. But, given some $g\in G$, we have that $(\mathcal{O}_H(e_i))^g = \mathcal{O}_{H^g}(e_i^g)$, which is equal to $\mathcal{O}_{H}(e_k)$ for some $k$, as $H$ is normal in $G$ and the $e_j$ form a $G$-orbit.

Finally, $f'_1 + \dots + f'_s$ is the sum of the elements of the $\mathcal{O}_H(e_i)$ for all $1\leq i\leq s$. But by the above we know that this is just the sum of $G$-conjugates of $e$, i.e. $\cpisum{e}{G} = f$.
\end{proof}

Recall the ideal correspondence (IV) of Definition \ref{defn: corresps 2}. Set $G_1$ to be the centraliser in $G$ of $f'_1$, and write $\mathfrak{N}$ and $Q$ for the prime ideals of $kG_1$ corresponding under (IV) to $\mathfrak{M}$ and $P$ respectively. Write also $\mu: f\cdot\overline{kG} \to M_s(f'_1\cdot\overline{kG_1})$ for the isomorphism of Lemma \ref{lem: MUL} obtained from the idempotent decomposition of $f$ given by Lemma \ref{lem: H-orbits are a mutually orthogonal G-orbit}.

\begin{propn}\label{propn: IV preserves control}
$P$ is controlled by $H$ if and only if $Q$ is controlled by $H_1$.
\end{propn}

\begin{proof}
By applying the isomorphism $\mu$, it is clear that $f\cdot\overline{P}$ is controlled by $f\cdot\overline{kH}$ if and only if $\mu(f\cdot\overline{P})$ is controlled by $\mu(f\cdot\overline{kH})$. But by Lemma \ref{lem: MUL}(ii), (iii), we may calculate $\mu(f\cdot\overline{P}) = M_s(f'_1\cdot\overline{Q})$ and $\mu(f\cdot\overline{kH}) = M_s(f'_1\cdot\overline{kH_1})$. Hence, by Morita equivalence (Proposition \ref{propn: corresps preserve control}(III)), $f\cdot\overline{P}$ is controlled by $f\cdot\overline{kH}$ if and only if $f'_1\cdot\overline{Q}$ is controlled by $f'_1\cdot\overline{kH_1}$. Now we are done by two applications of Proposition \ref{propn: corresps preserve control}(I).
\end{proof}

\subsection{Faithfulness}

Adopt Setup \ref{setup: k, G, bar}, and let $e\in \cpi(\mathfrak{M})$ and $f = \cpisum{e}{G}$. Write $e = e_1, \dots, e_r$ for the $G$-orbit of $e$ inside the ring $R = f\cdot\overline{kG}$, and note that they sum to $f = 1_R$. Write $q: kG\to f\cdot\overline{kG} = R$ for the natural quotient map.

Let $A$ be an ideal of $R$. Recall Definition \ref{defn: dagger}; we give an analogue for the current, slightly more general situation.

\begin{defn}\label{defn: dagger in peirce stuff}
Let $A$ be an ideal of $R$. Then we may define the \emph{$G$-unfaithful part} of $A$ as
$$A^\dagger = q^{-1}(A+1_R)\cap G = \ker(G\to (R/A)^\times).$$
\end{defn}

We show first that this is indeed the appropriate analogue of the earlier definition.

\begin{lem}\label{lem: relating old and new dagger}
Let $I$ be an ideal of $kG$ containing $\mathfrak{M}$, and recall the definition of $I^\dagger$ from Definition \ref{defn: dagger}. Set $R = f\cdot\overline{kG}$ and $A = q(I)\lhd R_i$, and write $A^\dagger$ for the subgroup of $G$ defined in Definition \ref{defn: dagger in peirce stuff}. Then $A^\dagger = I^\dagger$.
\end{lem}

\begin{proof}
If $g\in I^\dagger$, i.e. $g-1\in I$, then clearly $q(g-1)\in A$, so $g\in A^\dagger$. Conversely, suppose $g\in A^\dagger$, so that there exists some $x\in\ker q = (1-f)\overline{kG}$ with $g-1+x\in I$. But $I$ contains $\mathfrak{M} = (1-f)\overline{kG}$, so in particular $I$ also contains $x$, showing that $g-1\in I$.
\end{proof}

Now recall that $R_{ii} = e_i\cdot\overline{kG}\cdot e_i$ may be naturally identified with $e_i\cdot\overline{kG_i}$, where $G_i$ is the centraliser inside $G$ of $e_i\in R$. Write $q_i: kG_i\to e_i\cdot\overline{kG_i} = R_{ii}$ for the natural quotient map.

The following lemma shows the relationship between $A^\dagger$ and $B_i^\dagger$.

\begin{lem}\label{lem: relating A^dagger and B^dagger}
Let $A$ be an ideal of $R$, and take $A^\dagger$ to be its $G$-unfaithful part, as in Definition \ref{defn: dagger in peirce stuff}. Let $B_i = q_i(A)$ be an ideal of $R_{ii}$, and take $B_i^\dagger$ to be its $G_i$-unfaithful part likewise.

\begin{itemize}
\item[(i)] If $A\neq R$, then $A^\dagger \leq \bigcap_{i=1}^r G_i$.
\item[(ii)] If $A\neq R$, then $A^\dagger = \bigcap_{i=1}^r B_i^\dagger$.
\item[(iii)] The $\{B_i^\dagger\}$ are a $G$-orbit under conjugation.
\end{itemize}
\end{lem}

\begin{proof}
$ $

\begin{itemize}
\item[(i)]
Suppose not: then there exist some $i\neq j$ and some $g\in A^\dagger$ with $e_i q(g) = q(g) e_j$, so that $e_i q(g) e_i = 0$. Then
$$q(g-1) \in A \implies e_iq(g-1)e_i\in A \implies e_i\in A,$$
but then by conjugating by elements of $q(G)$ we see that $e_l\in A$ for all $l$, and so $1 = \sum_l e_l \in A$, so $A = R$, which is a contradiction.

\item[(ii)]
$\boxed{\leq}$ Fix $i$. If $q(g-1) \in A$, then $g\in G_i$ by part (i), so $e_iq(g-1)e_i\in B_i$.

$\boxed{\geq}$ Let $g\in \bigcap_i B_i^\dagger$, so that $e_i q_i(g) e_j = q_i(g) e_i e_j = 0\in A$ for all $i\neq j$, and $e_iq_i(g-1)e_i\in B_i$ $\subseteq A$ for all $i$. Then $q(g-1) = \sum_{i\neq j} e_i q(g) e_j + \sum_i e_iq_i(g-1)e_i \in A$.

\item[(iii)]
$A$ is $q(G)$-stable, so if $q(g)^{-1}e_iq(g) = e_j$ then $q(g)^{-1}B_iq(g) = q(g)^{-1}(e_i A e_i)q(g) = e_j A e_j = B_j$. Likewise, $g^{-1}G_ig = g^{-1}(e_i G e_i)g$ $ = e_j G e_j = G_j$. It follows from Definition \ref{defn: dagger in peirce stuff} that $g^{-1}(B_i^\dagger)g = B_j^\dagger$.\qedhere
\end{itemize}
\end{proof}

Combining these results, we get the following.

\begin{propn}\label{propn: IV preserves faithfulness}
Let $\mathfrak{M}\leq P\lhd kG$ and $\mathfrak{N}\leq Q\lhd kG_1$ be prime ideals which correspond under (IV). Then $P^\dagger = \bigcap_{g\in G} (Q^\dagger)^g$.
\end{propn}

\begin{proof}
Setting $A = q(P)$ and $B_i = q_i(P)$, we see from Lemma \ref{lem: relating old and new dagger} that $A^\dagger = P^\dagger$ and $B_1^\dagger = Q^\dagger$. Now the result follows from Lemma \ref{lem: relating A^dagger and B^dagger}.
\end{proof}

\textit{Proof of Theorem E.} Part (i) follows from Proposition \ref{propn: IV preserves primality}, part (ii) from Proposition \ref{propn: IV preserves control}, and part (iii) from Proposition \ref{propn: IV preserves faithfulness}.\qed

\section*{Acknowledgements}

This work was originally completed as part of my D.Phil. thesis at the University of Oxford. I would like to thank my supervisor, Konstantin Ardakov, without whose support this work would not have been possible. I would also like to express my gratitude to the University and to EPSRC for the funding that allowed me to complete this project; and to the University, Brasenose College, and Queen Mary University of London for their hospitality. I am also grateful to an anonymous reviewer for detailed feedback on the exposition of this material in an earlier draft, which helped me to significantly improve its presentation.

\bibliography{../../biblio/biblio-all}
\bibliographystyle{plain}

\end{document}